\documentclass{article}

\usepackage{graphicx}
\usepackage{caption}
\usepackage{xcolor} 
\usepackage{subcaption}
\usepackage{makeidx} 
\usepackage{enumitem}
\usepackage{amsfonts,dsfont,mathrsfs}
\usepackage{amssymb}
\usepackage{amsmath}
\usepackage[colorinlistoftodos]{todonotes}
\usepackage[left=2.5cm, right=2.5cm, top=2.5cm, bottom=3cm]{geometry}
\usepackage{caption}
\usepackage{color} 
\usepackage{bm,stackrel}
\usepackage{comment}
\usepackage[sort]{cite}
\usepackage{tikz,pgf}
\usepackage{amsopn}
\usepackage{graphicx}
\usepackage{epstopdf}
\usepackage{algorithm,algorithmic}

\usetikzlibrary{shapes.geometric}
\usetikzlibrary{arrows,patterns}
\usepackage{pgfplots}
\pgfplotsset{compat=newest}
\usepgfplotslibrary{patchplots}

\definecolor{greenExample}{rgb}{0.13, 0.55, 0.13}
\definecolor{antiquefuchsia}{rgb}{0.57, 0.36, 0.51}

\title{Existence of solutions for deterministic bilevel games under a general Bayesian approach}

\author{
	D. Salas\thanks{Instituto de Ciencias de la Ingenier\'ia, Universidad de O'Higgins, Rancagua, Chile. (david.salas@uoh.cl, anton.svensson@uoh.cl). The first author was partially funded by ANID-Chile under the grant \textit{FONDECYT postdoctorado 3190229} and by MathAmsud program trough the project MATHAMSUD 20-MATH-08. The second author was partially funded by ANID-Chile under the grant \textit{FONDECYT postdoctorado 3210735}.}
	\and 
	A. Svensson\footnotemark[2] 
}

\newlength\myindent
\newcommand\bindent{
	\begingroup
	\setlength{\itemindent}{\myindent}
	\addtolength{\algorithmicindent}{\myindent}
}
\newcommand\eindent{\endgroup}

\providecommand{\co}{\mathop{\rm co}\nolimits}
\providecommand{\cco}{\mathop{\rm \overline{co}}\nolimits}

\providecommand{\gph}{\mathop{\rm gph}\nolimits}
\providecommand{\dom}{\mathop{\rm dom}\nolimits}
\providecommand{\into}{\mathop{\rm int}\nolimits}

\providecommand{\diam}{\mathop{\rm diam}\nolimits}
\providecommand{\argmin}{\mathop{\rm argmin}}

\providecommand{\spn}{\mathop{\rm span}\nolimits}

\providecommand{\ri}{\mathop{\rm ri}\nolimits} 
\providecommand{\aff}{\mathop{\rm aff}\nolimits} 
\providecommand{\dim}{\mathop{\rm dim}\nolimits} 
\providecommand{\adim}{\mathop{\rm adim}\nolimits}

\let\epsilon\varepsilon
\DeclareMathOperator{\supp}{supp}
\DeclareMathOperator{\cl}{cl}

\providecommand{\tto}{\mathop{\rightrightarrows}\nolimits}

\providecommand{\norm}{\mathop{ \|\cdot\| }\nolimits}

\providecommand{\Liminf}{\mathop{\rm Liminf}\limits}
\providecommand{\Limsup}{\mathop{\rm Limsup}\limits}
\providecommand{\Lim}{\mathop{\rm Lim}\limits}

\newcommand{\R}{\mathbb{R}}

\newcommand{\N}{\mathbb{N}}

\newcommand{\E}{\mathbb{E}}

\newcommand{\Ball}{ \mathbb{B}}

\newcommand{\proj}{\mathrm{proj}}

\newcommand{\ind}{\mathds{1}}
\newcommand{\centroid}{\mathfrak{c}}

\usepackage{amsmath,amsfonts,amssymb,dsfont,mathrsfs,amsthm}
\numberwithin{equation}{section}
\numberwithin{figure}{section}

\newtheorem{remark}{Remark}
\newtheorem{definition}{Definition}
\newtheorem{example}{Example}
\newtheorem{theorem}{Theorem}
\newtheorem{lemma}{Lemma}
\newtheorem{proposition}{Proposition}
\newtheorem{corollary}{Corollary}

\begin{document}
	
	\maketitle
	
	\begin{abstract}
		In 1996, Mallozzi and Morgan \cite{Mallozzi1996} proposed a new model for Stackelberg games which  we refer here to as the Bayesian approach. The leader has only partial information about how followers select their reaction among possibly multiple optimal ones. This partial information is modeled as a decision-dependent distribution, the so-called belief of the leader. In this work, we formalize the setting of this approach for bilevel games admitting multiple leaders and we provide new results of existence of solutions. We pay particular attention to the fundamental case of linear bilevel problems, which has not been studied before, and which main difficulty is given by possible variations in the dimension of the reaction set of the follower. Our main technique to address this difficulty is based on a stronger notion of continuity for set-valued maps that we call rectangular continuity, and which is verified by the solution set of parametric linear problems. Finally, we provide some numerical experiments to address linear bilevel problems under the Bayesian approach.
	\end{abstract}
	
	
	
	
	\section{Introduction}\label{sec:Intro}
	
	The term \emph{bilevel games} is here used to refer to non-cooperative models ranging from bilevel programming problems \cite{dempe2002foundations} to multi-leader-follower games \cite{aussel2020stateof,HuFukushima2015}. They were first introduced by H. von Stackelberg \cite{von1934market} in the 30's to model hierarchical decision problems, and since the last few decades they have attracted the attention of many researchers, due to their large potential in applications. 
	Bilevel games consist in a non-cooperative interaction between two or more agents, which are divided in two groups: the leaders and the followers. The main difference between these two groups of agents is that, on the one hand, leaders know how followers make their decisions, and so they are able to include this information in their own decision-making processes. On the other hand, followers are only capable of observing the decisions of the leaders once they have been made, and then react accordingly. 
	
	It is well known that, unless the followers' reactions are unique, the problems of the leaders are ill-defined. To solve this issue, the most common approaches that we find in the literature are the \textit{Optimistic approach} and the \textit{Pessimistic approach} \cite{dempe2002foundations,liu2018pessimistic}. A leader is optimistic, if he or she assumes that the followers will provide the reaction, among the optimal ones, that favors him or her the most. A leader is pessimistic, if he or she assumes that the followers will provide the reaction that harms him or her the most.
	
	In many applied situations, however, followers are neither adversarial nor cooperative. Thus, some other approaches have been proposed as alternatives between the optimistic/pessi\-mis\-tic dichotomy, like using selections as single-valued reaction of the follower (see, e.g., \cite[pp. 6]{DempeEtAl2015}), or moderate approaches that consider convex combinations between the min/max values over the followers' variables \cite{alves2019new,jia2011new}. 
	
	We will focus our attention in a more general approach that was first introduced by L. Mallozzi and J. Morgan in \cite{Mallozzi1996} for bilevel games with one leader and which we refer here to as the Bayesian approach (see Remark \ref{rem:renaming}). The approach works as follows: Each leader is endowed with a private belief about how followers select their reaction. This belief is modeled as a probability distribution over the set of optimal reactions. Then, leaders play a non-cooperative game, where each leader's payoff function depends on his or her decision, the decisions of the other leaders, and his or her private belief.
	
	Clearly, the Bayesian approach encloses both Optimistic and Pessimistic approaches, since optimism and pessimism can be modeled as beliefs. However, the price to pay is that beliefs are complex
	objects to work with, as they are decision-dependent distributions \cite{JonsbratenWetsWoodruff1998}. 
	Programming with decision-dependent uncertainty has been explored recently in several works \cite{ahmed2000strategic,peeta2010pre,nohadani2018optimization,GoelGrossmann2006,drusvyatskiy2020stochastic}, and it is an active field of research due to its challenging difficulties, both from the theoretical and algorithmic point of view.

	\subsection*{Literature review}
	In \cite{Mallozzi1996}, the analysis of the Bayesian approach for bilevel games (with a single leader) is focused on the existence of solutions in two cases: first, when the reactions of the follower are finite, and second, when the values of the reaction map have positive Lebesgue measure (which is equivalent to have nonempty interior if they are convex).
	The second case has been applied to the so-called regularized bilevel problem, where the reaction map consists of approximate solutions of the lower-level problem being a convex parametric optimization problem. In \cite{MallozziMorgan2005}, the above results have been adapted to reaction maps whose values have either nonempty interior or consist of a single point. 
	
	However, none of the results of \cite{Mallozzi1996,MallozziMorgan2005} can be applied to bilevel problems with reaction maps of varying affine dimension. In particular, they cannot be applied (in general) to linear bilevel problems with exact solutions on the lower-level. 
	
	To the best of our knowledge, there is no other contribution in the literature about the Bayesian approach, beyond the aforementioned works. This is somehow confirmed by the recent survey \cite{caruso2020regularization}.
	
	\subsection*{Our contributions}
	

	We first formalize a natural extension of the Bayesian approach to bilevel games, admitting multiple leaders. This is done in the Preliminaries (Section \ref{sec:Pre}).
	In Section \ref{subsec:GeneralExistence}, we provide a general sufficient (and necessary, in some sense) condition to have existence of solutions for the Bayesian approach, which is the weak continuity of the beliefs in the sense of measure-valued maps. Weak continuity allows us to guarantee existence of solutions for the case of one leader (Theorem \ref{thm:ExistenceSingleLeader}), and existence of mixed equilibria for the case of multiple leaders (Theorem \ref{thm:ExistenceMixedEq}).
	
	In Section \ref{subsec:RecContinuity}, we introduce a new notion of continuity, called \textit{rectangular continuity}. Intuitively, a rectangularly continuous set-valued map might change the affine dimension of its values, but in a controlled fashion. We show, in Section \ref{subsec:ConstructionBeliefs}, that if the reaction map of the followers in a bilevel game has convex values and it is rectangularly continuous, then all the beliefs coming from a continuous strictly positive density are weak continuous. This is our first main result, enclosed in Theorem \ref{thm:WeakContinuityUnderRecContinuity}.
	
	Finally, in Section \ref{subsec:linearrectangular}, we show that the solution maps of parametric linear problems of the form $x\mapsto \min_y\{ \langle c,y\rangle\ :\ Ax+By\leq b\}$ are not only continuous (which is well known), but they are rectangularly continuous. In fact, we prove this result for separable parametric problems given by weakly analytic functions in Theorem \ref{thm:RecContinuityWeakAnalytic}, for which parametric linear problems are a particular case. This is our second main contribution. Our results are new even in the case of linear bilevel programming, and this is summarized in Corollary \ref{cor:ExistenceLinear}.
	
	In Section \ref{sec:Computability}, we provide some numerical experiments for bilevel programming under the Neutral approach, which is the Bayesian approach with the belief induced by uniform distributions. For this numerical tests, we implement the differential evolution metaheuristic \cite{DiffEvolutionMatlab} and compute the Bayesian objective function through a Monte-Carlo method. We use the BOLIB library \cite{zhou2020bolib} as benchmark problems.
	
	
	\section{Preliminaries and model formulation}\label{sec:Pre}
	
	Throughout this article, we 
	gather three main fields of mathematics that we assume the reader is
	acquainted with in a basic level: (Bilevel) Game theory \cite{fudenberg1991game,Tadelis2013,dempe2002foundations}, Set-valued analysis \cite{AubinFrankowska1990,rockafellar2009variational} and Measure theory \cite{klenke2013probability}. 
	
	\subsection{Notation}\label{subsec:Notation}
	We will work on finite-dimensional spaces, namely $\R^d$, endowed with the Euclidean norm $\norm$ and the inner product $\langle \cdot,\cdot\rangle$. An open ball of radius $r>0$ and centered at $x\in \R^d$ is denoted by $B(x,r)$, and we write $B(r)$ if the ball is centered at the origin. The closed unit ball centered at the origin will be denoted by $\mathbb{B}$. Given a subset $A$ of $\R^d$, we denote by $\cl(A)$, $\into(A)$, $\partial A$, $\ri(A)$ and $\aff(A)$, its closure, interior, boundary, relative interior and affine hull, respectively. The affine dimension of $A$, that is, the dimension of $\aff(A)$, is denoted by $\adim(A)$. We denote by $\co(A)$ and $\cco(A)$ the convex hull and the closed convex hull of $A$, respectively. The orthogonal set to $A$ is the set 
	\[
	A^{\perp}:=\{x\in \R^d: \langle x,a\rangle =0,\, \forall a\in A\}.
	\] 
	Given sets $A,B\subset \R^p$ and $\eta\in\R$ we write $A+B$ and $\eta A$ to denote the Minkowski addition and scalar multiplication, respectively.
	
	
	For two points $x,y\in \R^d$, we write $[x,y] = \{ tx + (1-t)y\ :\ t\in [0,1] \}$ to denote the closed segment enclosed by $x$ and $y$. Similarly, we write $(x,y)$, $[x,y)$ and $(x,y]$, for the open segment, and the semi-open segments, respectively. For a nonempty closed convex set $C\subset\R^d$ and a point $x\in \R^d$, we denote by $\proj_C(x)$ or by $\proj(x;C)$ the metric projection of $x$ onto $C$. If $C=A\times B$ is a convex closed product set on the product space $\R^d\times\R^p$, we also denote by $\proj_A(\cdot)$ the parallel projection onto $A$, by identifying $A$ with $A\times\{ 0\}$.  
	
	
	Given two nonempty sets $X\subset \R^d$ and $Y\subset \R^p$, and a set-valued map $S:X\tto Y$, 
	we denote the (effective) domain of $S$  by $\dom S =\{ x\in X\ :\ S(x)\neq\emptyset\}$, and the range of $S$ by $S(X) = \{ y\in Y\ :\ \exists x\in X,\, y\in S(x) \}$. The graph of $S$ is denoted by $\gph S = \{(x,y)\in X\times Y\ :\ y\in S(x)\}$.

	Let $Y\subset\R^p$ be a nonempty closed set. We denote by $\mathcal{B}(Y)$ the Borel $\sigma$-algebra of $Y$, that is, the $\sigma$-algebra generated by the topology in $Y$. We recall that a measure $\nu$ defined over $\mathcal{B}(Y)$ is called a Borel measure.
	Given an integer $k\in \{0,\ldots,p\}$, we will denote by $\lambda_k$ the $k$-dimensional Hausdorff measure. It is well-known that for every affine subspace $H$ of $\R^p$ of dimension $k$, $\lambda_k$ coincides with the usual Lebesgue measure on $H$ (see, e.g., \cite{EvansGariepy2015}). From now on, we will simply refer to $\lambda_k$ as the $k$-dimensional Lebesgue measure. In the full dimensional case ($k=p$) we omit the subindex, simply writing $\lambda$.
	
	The centroid (also known as the barycenter) of a nonempty bounded convex set $K\subset \R^p$ is given by
	\begin{equation}
		\centroid(K) := \frac{1}{\lambda_k(K)}\int_K yd\lambda_k(y),
	\end{equation}
	where $k=\adim(K)$. Observe that, since $\lambda_k(K)$ is strictly positive for convex sets, the centroid is always well-defined.
	
	We denote by $\mathscr{P}(Y)$ the family of all (Borel) probability measures on $Y$, endowed with the topology of weak convergence. Recall that a sequence $(\nu_n)\subset \mathscr{P}(Y)$ weak converges to $\nu\in \mathscr{P}(Y)$, which is denoted by $\nu_n\xrightarrow{w}\nu$, if
	\[
	\int_Y f(y)d\nu_n(y) \to \int_Y f(y) d\nu(y),
	\]
	for all continuous functions $f:Y\to \R$ with bounded support. It is well known that $\mathscr{P}(Y)$ is metrizable for the induced topology by the weak convergence and, whenever $Y$ is compact, then $\mathscr{P}(Y)$ is compact as well (see, e.g., \cite[Chapter 13]{klenke2013probability}). 
	Accordingly, we will say that a measure-valued map $h:X\to \mathscr{P}(Y)$ is weak continuous if for every sequence $(x_n)\in X$ converging to a point $x\in X$, the sequence $(h(x_n))\subset \mathscr{P}(Y)$ weak converges to $h(x)$.
	
	Now, let us recall a sequence of equivalences for the weak convergence on $\mathscr{P}(Y)$, known as the Portemanteau Theorem (see, e.g., \cite[Thm. 13.16]{klenke2013probability}, and \cite[Excercise 2.6]{Billingsley1999}).
	
	\begin{theorem}[Portemanteau]\label{thm:Portemanteau} Let $Y$ be a nonempty closed subset of $\R^p$, $(\nu_n)$ be a sequence of probability measures of $\mathscr{P}(Y)$ and let $\nu\in \mathscr{P}(Y)$. The following assertions are equivalent:
		\begin{enumerate}\setlength{\itemsep}{0.3cm}
			\item $\nu_n\xrightarrow{w}\nu$. 
			\item $\liminf \int_Y fd\nu_n \geq \int_Y fd\nu$, for every lower semicontinuous function $f:Y\to\R$ bounded from below.
			\item $\limsup \nu_n(C) \leq \nu(C)$, for every closed set $C$ in $Y$.
			\item $\liminf \nu_n(U) \geq \nu(U)$, for every open set $U$ in $Y$. 
		\end{enumerate}
	\end{theorem}

	
	\subsection{Bilevel games and the Bayesian approach}\label{subsec:Model}
	
	
	Formally, a bilevel game is posed as follows. In the upper level, there is a finite set $\mathfrak{I}$ of players, usually known as the leaders, interacting in a non-cooperative Generalized Nash Equilibrium Problem (GNEP for short, see, e.g., \cite{facchinei2010generalized}). A GNEP, is an equilibrium problem where the feasible strategies of each player might depend on the strategies of the other players. Thus, each leader $i\in \mathfrak{I}$ has a decision variable $x_i \in K_i(x_{-i}) \subset X_i\subset \R^{d_i}$ and a function $\theta_i$ that he or she wishes to minimize, where we use the usual notation $x_{-i}\in X_{-i} = \prod_{j\in\mathfrak{I}\setminus\{i\}}X_j$, to denote the other players' strategy.
	
	Given a leaders' decision vector $x\in X=\prod_{i\in\mathfrak{I}}X_i\subset \R^d$,  the lower-level is represented by a set-valued map $S:X\tto Y\subset \R^p$ and a decision variable $y\in S(x)$. This decision variable $y$ is chosen by a finite set of independent agents, called the followers, who react to the leaders' decision vector $x$. The set $S(x)$ represents the solution set (or approximate solution set) of the followers' problem which can be given by a parametric optimization problem \cite{bank1982non}, a parametric game \cite{aussel2019towards}, a parametric variational inequality \cite{mansour2008quasimonotone}, etc. 
	
	We will call $S$ the \emph{reaction map}, and it will play a central role in our analysis.	We refer to $Y$ as the \emph{ambient space}, and if it is not explicitly defined, one simply takes $Y = \cco(S(X))$. In general, we will assume $Y$ to be compact.  Finally, each leader aims to solve the following problem:  
	
	\begin{equation}\label{eq:LeaderProblem}
		\mathcal{L}_i(x_{-i}) = \begin{cases}
			\displaystyle\min_{x_i}\,\,& \theta_i(x_i,x_{-i},y)\\
			\mbox{s.t.}\quad &x_i \in K_i(x_{-i}),\, y \in S(x).
		\end{cases}
	\end{equation}
	In what follows, unless stated otherwise, we will always assume for each $i\in\mathfrak{I}$ that $\theta_i:X\times Y\to \R$ is continuous, $X_i$ is nonempty and compact and that the set-valued map $K_i:X_{-i}\tto X_{i}$ has closed graph. We will also assume that $S$ has closed graph, which in particular implies that $S$ has closed values (and hence also measurable).
	
	
	Note that for some point $x\in X$, one might have that $S(x) = \emptyset$. In such a case, for each $i\in \mathfrak{I}$ we consider the decision $x_i$ to be infeasible for problem $\mathcal{L}_i(x_{-i})$, even if $x_i\in K_i(x_{-i})$. Under the conventional abuse of notation that $x=(x_i,x_{-i})$, 
	we can always replace $K_i(x_{-i})$ by $K_i(x_{-i})\cap\{  x_i\in X_i\ :\ (x_i,x_{-i})\in \dom S(x) \}$ for every $i\in \mathfrak{I}$ and every $x_{-i}\in X_{-i}$. Therefore
	we will always assume that 
	\begin{equation}\label{eq:NomeptyAssumptionGeneral}
		\gph K_i \subset \dom S, \quad\forall i\in\mathfrak{I}.
	\end{equation}
	This assumption, however, yields that a bilevel game has an intrinsic GNEP nature due to the possible empty values of $S$, even if originally we considered $K_i(x_{-i}) = X_i$ for each $i\in\mathfrak{I}$ and for each $x_{-i}\in X_{-i}$. This motivates the following distinction: 
	\begin{definition}[Plain bilevel games]
		We say that a bilevel game given by problems \eqref{eq:LeaderProblem} is \emph{plain} if $K_i(x_{-i}) = X_i$ for each $i\in\mathfrak{I}$ and for each $x_{-i}\in X_{-i}$, and $X\subset \dom S$.
	\end{definition}
	Observe that if the bilevel game is plain, then the leaders interact in a Nash equilibrium problem rather than in a GNEP.
	
	If only one leader is involved, we omit the subindex $i\in \mathfrak{I}$, by considering the problem $\mathcal{L}=\min_x\{ \theta(x,y)\ :\ x\in X,\, y\in S(x) \}$. In this setting, if $S(x)$ is the solution set of a parametric optimization problem, the bilevel game is known as a bilevel programming problem (see, e.g., \cite{dempe2002foundations}). The general form of a bilevel programming problem is given by
	\begin{equation}\label{eq:BilevelProgrammingProblem}
		BP:=\begin{cases}
			\displaystyle\min_{x}\,\, &\theta(x,y)\\
			\mbox{s.t.}\quad &x \in X,\\
			& y \text{ solves } \begin{cases}
				\min_{y}\,\,& f(x,y)\\
				\text{s.t.} &y\in C(x),
			\end{cases}
		\end{cases}
	\end{equation}
	where $C:X\tto Y\subset \R^p$ is the constraints set-valued map of the follower, usually of the form $C(x) = \{y\in Y\ :\ g(x,y)\leq 0\}$, where $g:X\times\R^p\to \R^q$ is a continuous vector-valued function and $f:X\times \R^p\to\R$ is continuous. The reaction map $S:X\tto Y$ is then given by $S(x) = \argmin_y\{ f(x,y)\ :\ y\in C(x)\}$.
	Observe that in this case, assuming that the ambient space $Y$ is compact, one has by Weierstrass Theorem that $\dom S = \dom C$, that is, $S(x)$ is nonempty whenever the lower-level problem is feasible. Thus, in this case, by replacing $X$ by $X\cap \dom C$, we can always assume that the lower-level is feasible within the feasible set $X$. In other words, a bilevel game with only one leader can always be assumed to be plain. 
	
	\begin{remark} In this work, we are omitting the possibility of coupling constraints for the leaders, that is, constraints of the form $G(x,y)\leq 0$ in the upper level, which are blind for the followers (see, e.g., \cite{mersha2006linear,audet2006note}). An interesting extension of our work is to adapt the Bayesian approach for this setting as well.
	\end{remark}


	As explained in the introduction, the reaction map $S$ might not be single-valued at every point, so the bilevel problem carries an ambiguity that must be handled. In the literature the most common approaches are the Optimistic and Pessimistic ones (see, e.g., \cite{liu2018pessimistic,dempe2002foundations}). In the rest of this section, we recall the Bayesian approach, which solves the ambiguity in a more general form by considering beliefs for the leaders. We start by providing a formal definition of what a belief over a set-valued map is.
	
	\begin{definition}[Belief]\label{def:Belief} Let $D$ be a nonempty subset of $\R^d$, $Y$ be a nonempty closed subset of $\R^p$, and $S:D\tto Y$ be a set-valued map with nonempty measurable values (i.e. $S(x)\in \mathcal{B}(Y)\setminus\{\emptyset\}$ for all $x\in D$). A \emph{belief over $S$} is a mapping $\beta: D\to \mathscr{P}(Y)$ assigning to each $x\in D$ a probability measure $\beta_x$ on $Y$ such that
		\begin{equation}\label{eq:ConditionWellDefined}%
			\forall x\in D,\qquad \beta_x(S(x)) = 1.
		\end{equation}
		A finite collection of beliefs $\{ \beta^i:D\to\mathscr{P}(Y)\ :\ i\in \mathfrak{I} \}$ is called a belief system over $S$. 
	\end{definition}
	
	In the case of a general bilevel game \eqref{eq:LeaderProblem}, we consider $D= \dom S$ (or $D = \bigcup_{i\in\mathfrak{I}} \gph K_i$, according to \eqref{eq:NomeptyAssumptionGeneral}). For a plain bilevel game (such as the case with a single leader), we set $D=X$.
	
	
	
	If a leader $i\in \mathfrak{I}$ has a belief $\beta^i$ over $S$, it means that for a given leaders' decision vector $x\in D\subset \R^d$, he or she predicts that the reaction of the followers is a random variable with probability distribution $\beta^i_x$. Accordingly, he or she aims to minimize the expected value of his/her cost function $\theta_i$, with respect to $\beta^i$. Thus, problem \eqref{eq:LeaderProblem} can be formulated as follows:
	\begin{equation}\label{eq:LeaderProblem-Stochastic}
		\mathcal{L}_i(x_{-i},\beta^i) = \begin{cases}
			\displaystyle\min_{x_i}\,\, &\E_{\beta^i_x}(\theta_i(x_i,x_{-i},y)):=\displaystyle\int_Y \theta_i(x_i,x_{-i},y) d\beta_x^i(y)\\
			\mbox{s.t.}\quad &x_i \in K_i(x_{-i})
		\end{cases}
	\end{equation} 
	
	Note that, since the belief $\beta^i$ is fixed (in the sense that it is not decided within problem \eqref{eq:LeaderProblem-Stochastic} but prior to it), the equilibrium problem given by \eqref{eq:LeaderProblem-Stochastic} becomes a generalized Nash equilibrium problem (GNEP), where all the complexity of the bilevel structure is captured by the beliefs. In this GNEP  the objectives of the players are no longer random, and we will call it a 
	\emph{bilevel game with belief system} $\{ \beta_i\ :i\in \mathfrak{I} \}$. 
	
	The \textit{Bayesian approach} consists in selecting a belief system, prior to the problem, that models the predictions of the leaders to the uncertain followers' reactions, and then to pose the problem as the bilevel game with the selected belief system. 
	
	\begin{remark}\label{rem:renaming}
		The renaming of the approach as Bayesian is motivated by several aspects: First, it encompasses the fact that leaders must form beliefs over the (unknown) behavior of the followers which fits the Bayesian interpretation of probabilities (see, e.g., \cite{wasserman2004all}). Secondly, these beliefs can be understood as priors, that could be refined after playing the game. While this is not treated in this work, the idea opens the possibility of ``learning'' the real behavior of the followers by updating the belief and observing outcomes of the game. 
		Finally, considering that the selection $y\in S(x)$ is unknown due to lack of information,  one might understand it as the type of the follower, following the ideas of Bayesian games \cite{harsanyi1967games}.    
	\end{remark}
	
	
	\subsection{Scope of the work}\label{subsec:Scope}
	The present work is 
	about existence of solutions of Problem \eqref{eq:LeaderProblem-Stochastic}.  
	The scope of this article is highlighted by the following assumptions.
	\paragraph{(A) We will consider only plain bilevel games} If the bilevel game is not plain, Problem \eqref{eq:LeaderProblem-Stochastic} is a GNEP. Positive results for existence of  
	equilibria for GNEPs are very few and they rely on either some generalized convexity condition or ideas coming from potential games. These conditions are generally not reasonable for bilevel games.
	Furthermore, up to the best of our knowledge, there is no good notion of mixed equilibrium for GNEP. Thus, we will restrain ourselves to the setting where Problem \eqref{eq:LeaderProblem-Stochastic} is a NEP, which are plain bilevel games (including general single-leader games).
	
	\paragraph{(B) We will assume that $S$ takes convex values} While this is a simplifying assumption, the cases where $S$ takes convex values encompasses several important settings such as the case where $S(x)$ is the solution of a parametric convex programming problem, or the case where $S(x)$ is the set of correlated equilibria of a Nash game with finite strategy sets  (see, e.g., \cite{Papadimitriou2007Complexity}). In particular, this is the setting of bilevel games where the lower-level is given by a parametric linear problem of the form $x\mapsto \min_y\{ \langle c,y\rangle\ :\ Ax + By \leq b \}$, which is our main motivation.
	
	\paragraph{(C) We will assume that the beliefs $\beta$ are nonatomic}  
	Our main results are focused on beliefs $\beta:x\in X\mapsto \beta_x \in \mathscr{P}(Y)$ that are absolutely continuous with respect to the (corresponding) Lebesgue measures. Specifically, we mean that for every $x\in X$ and every $A\subset S(x)$, 
	\[
	\lambda_{d_x}(A)=0\implies \beta_x(A) = 0,
	\]
	where $d_x=\adim(S(x))$ and $\lambda_{d_x}$ is the Lebesgue measure over $\aff(S(x))$. 
	
	An important belief that we shall consider in this setting (and that will lead our efforts in the following sections) is the one given by uniform distributions over the sets $S(x)$. This belief models the situation where the leader has no information over how the followers select their reaction, hence assigning equal density to each possible one. 
	
	\begin{definition}[Neutral belief] 
		\label{def:neutralbelief} Let $Y$ be a nonempty closed subset of $\R^p$ and let $S:X\tto Y$ be a set-valued map with nonempty convex compact values. The Neutral belief $\iota:X\to\mathscr{P}(Y)$ over $S$ is given by
		\[
		\iota_x(A) = \frac{\lambda_{d_x}(A\cap S(x))}{\lambda_{d_x}(S(x))},\quad A\in \mathcal{B}(Y),
		\]
		where ${d_x} = \adim(S(x))$ and $\lambda_{d_x}$ is the $d_{x}$-dimensional Lebesgue measure. 
	\end{definition}

	The Neutral approach is then given as the equilibrium problem \eqref{eq:LeaderProblem-Stochastic} where each leader minimizes the value function
	\begin{equation}
		\label{eq:LeaderProblem-neutral}
		\varphi_i
		(x_i,x_{-i}) =\E_{\iota_x}(\theta(x_i,x_{-i},\cdot))= \displaystyle \int_{S(x)}\frac{\theta_i(x_i,x_{-i},y)}{\lambda_{d_x}(S(x))}d\lambda_{d_x}(y).
	\end{equation}
	
	
	\section{Existence of solutions for the Bayesian approach}\label{sec:Existence}
	
	This section is devoted to study the existence of solutions for the Bayesian approach, based on sufficient conditions related to the reaction map $S$ and the belief $\beta$ over $S$. 
	

	\subsection{Existence results for weak continuous beliefs}\label{subsec:GeneralExistence}
	Here we will show two existence results for the problem defined by \eqref{eq:LeaderProblem-Stochastic} under the hypothesis that the 
	belief of each leader is weak continuous. 
	
	\begin{theorem}\label{thm:ExistenceSingleLeader} Let $Y$ be a nonempty compact set of $\R^p$ and let us consider a bilevel game given by \eqref{eq:LeaderProblem-Stochastic} with only one leader (i.e., $|\mathfrak{I}|=1$) with belief $\beta$ over $S:X\tto Y$. If one has that
		\begin{enumerate}
			\item[(i)]the cost function $\theta$ is lower semicontinuous,
			\item[(ii)] the constraints set $X$ is nonempty and compact, and
			\item [(iii)] the belief $\beta: X\to\mathscr{P}(Y)$ is weak continuous,
		\end{enumerate}
		then, $x\mapsto \E_{\beta_x}(\theta(x,\cdot))$ is lower semicontinuous and so, the (single leader) bilevel game admits at least one solution in the Bayesian approach.
	\end{theorem}
	
	\begin{proof}
		We will only show the lower semicontinuity of $x\mapsto \E_{\beta_x}(\theta(x,\cdot))$ since the existence of solutions will trivially follow from the compactness of $X$. Let $(x^n)\subset X$ be a sequence converging to a point $x\in X$.
		
		Fix $\varepsilon>0$. Since $\theta$ is lower semicontinuous, for each $y\in Y$ there exists a neighborhood $U_y\times V_y$  
		of $(x,y)$ such that $\diam(V_{y})\leq \varepsilon$ and
		\[
		\forall (x',y')\in (U_y\times V_y) \cap (X\times Y),\quad \theta(x',y')> \theta(x,y)-\varepsilon.
		\]
		Without loss of generality, we may choose $U_{y}\times V_{y}$ to be closed. 
		Now, since $Y$ is compact, there exists a finite cover $\mathcal{V}_{\varepsilon}=\{V_{k}\}_{k=1}^m$ of $Y$, included in $\{V_{y}\ :\  y\in Y\}$. Let $\{y_{k}\}_{k=1}^m$ be a finite sequence such that $V_k=V_{y_k}$. 
		Consider the neighborhood of $x$ given by $U:=\bigcap_{k=1}^m U_{y_k}$. Note that $U$ satisfies
		\[
		\forall (x',y')\in (U\cap X)\times (V_{k}\cap Y),\quad  \theta(x',y')> \theta(x,y_k)-\varepsilon, 
		\]
		for each $k\in \{1,\ldots,m\}$. Let us define the function $\varphi_{\varepsilon}:Y\to \R$ given by
		\[
		\varphi_{\varepsilon}(y) = \min\{ \theta(x,y_k)-\varepsilon\ :\ y\in V_k\}. 
		\]
		
		By construction, $\varphi_{\varepsilon}$ is lower semicontinuous, bounded from below and satisfies $\varphi_{\varepsilon}\leq \theta(x',\cdot)$ for any $x'\in U$.
		In particular, for $n\in\N$ large enough $x^n\in U$, and hence, for every $y\in Y$, we have $\theta(x^n,y) \geq \varphi_{\varepsilon}(y)$. This yields, by the Portemanteau Theorem (see Theorem \ref{thm:Portemanteau}), that
		\begin{align*}
			\liminf_n \E_{\beta_{x^n}}(\theta(x^n,y)) &= \liminf_n\int_Y \theta(x^n,y) d\beta_{x^n}(y)\\
			&\geq \liminf_n\int_Y \varphi_{\varepsilon}(y)d\beta_{x^n}(y)\geq \int_Y \varphi_{\varepsilon}(y)d\beta_{x}(y).
		\end{align*}
		
		Now, let $\varepsilon_j = \frac{1}{j}$ and let $\varphi_j :=\varphi_{\varepsilon_j}$. We claim that $\varphi_j$ converges pointwise to $\theta(x,\cdot)$. Indeed, choose $y\in Y$. First, since $\varphi_j \leq \theta(x,\cdot)$ for every $j\in \N$, it is clear that $\limsup_j \varphi_j(y)\leq \theta(x,y)$. 
		
		By construction, for each $j\in \N $, there exists $y_j$ such that
		\[
		\varphi_j(y) = \theta(x,y_j)-1/j\quad\text{and}\quad\|y_j-y\|\leq 1/j.
		\]
		Then, we can write
		\begin{align*}
			\liminf_j \varphi_j(y) &= \liminf_j \left(\theta(x,y_j)-1/j\right)\geq \theta(x,y),
		\end{align*}
		where the last inequality follows from the fact that $\theta$ is lower semicontinuous and that $y_j\to y$. Thus, we have shown that
		\[
		\theta(x,y) \leq \liminf_j \varphi_j(y) \leq \limsup_j \varphi_j(y)\leq \theta(x,y),
		\]
		proving that $\varphi_j(y)\to \theta(x,y)$. and, by arbitrariness of $y\in Y$, proving our claim. 
		
		Now, since $\theta$ attains its minimum over $X\times Y$, we have that for every $j\in \N$, $\varphi_j \geq \min_{X\times Y} \theta -1$, and so, applying the Fatou Lemma for the measure $\beta_x$, we can write that
		\begin{align*}
			\liminf_n \E_{\beta_{x^n}}(\theta(x^n,\cdot)) &\geq \liminf_j \int_Y \varphi_j(y)d\beta_x(y)\\ 
			&\geq \int_Y \theta(x,y)d\beta_{x}(y) = \E_{\beta_{x}}(\theta(x,\cdot)).
		\end{align*}
		Since $(x^n)$ was an arbitrary sequence, we have shown the lower semicontinuity of the mapping $x\mapsto \E_{\beta_x}(\theta(x,\cdot))$, which finishes the proof. 
	\end{proof}

	\begin{remark}
		The weak continuity of the belief is used as a sufficient condition in Theorem \ref{thm:ExistenceSingleLeader} to prove the lower semicontinuity of the Bayesian  objective function of the leader given by
		\[
		x\mapsto \E_{\beta_x}(\theta(x,\cdot)).
		\]
		Conversely, if we require that the Bayesian objective function of the leader to be lower semicontinuous for any lower semicontinuous objective function $\theta$, then the belief must be weak continuous. This follows directly from part 2 of the Portemanteau Theorem by considering a lower semicontinuous function $f:Y\to\R$ and the leader's objective $\theta(x,y)=f(y)$. In this sense, weak continuity of the belief is also necessary.
	\end{remark}
	

	When multiple players are involved, the Bayesian approach suffers the same drawback as the rest of the approaches: it is a nonconvex generalized Nash equilibrium problem. While sufficient conditions to have existence of equilibria have been provided in the literature as generalized convexity conditions (see e.g. \cite{cotrina2020existence,bianchi2004equilibrium}), these are generally not satisfied in the context of bilevel games (see \cite[Example 4]{pang2005quasi}). Nevertheless, leaders can randomize their decisions in the setting of plain bilevel games. Then we can derive existence of mixed equilibria (introduced by Nash \cite{Nash1951} in the case of finite strategy sets, and extended to compact strategy sets one year later by Glicksberg \cite{Glicksberg1952}). In the next definition, we recall the notion of Mixed Equilibrium for compact strategy sets.
	
	\begin{definition}[Mixed equilibrium]\label{def:MixedEq} Let us consider a Nash Equilibrium problem  with $\mathfrak{I}=\{1,\ldots,N\}$ players, where each player $i\in \mathfrak{I}$ aims to solve
		\[
		\min_{x_i\in X_i} \varphi_i(x_i,x_{-i}),
		\]
		and all the objective functions $\varphi_i: \prod_{k=1}^N X_k\to \R$ are measurable and bounded from below. A mixed strategy of player $i\in \mathfrak{I}$ is a 
		probability measure $\nu_i \in \mathscr{P}(X_i)$. An $N$-tuple $(\nu_1,\ldots,\nu_N)$ in $\mathscr{P}(X_1)\times\cdots\times \mathscr{P}(X_N)$ is called a mixed equilibrium, if for each player $i\in \mathfrak{I}$,
		\[
		\nu_i\mbox{ solves } \min_{\nu\in \mathscr{P}(X_i)} \int\limits_{X_i}\int\limits_{X_{-i}} \varphi_i(x_i,x_{-i}) d\nu_{-i}(x_{-i}) d\nu(x_i),
		\]
		where $\nu_{-i}$ stands for the product measure $\displaystyle\bigotimes_{j\neq i}\nu_j$ on $X_{-i}$.
	\end{definition}
	
	The concept of mixed equilibria has been explored for bilevel games with continuous strategy sets in the setting of electricity auctions in \cite{pereda2019modeling}. Recent efforts have been made to provide methods to compute them for Nash equilibrium problems with continuous strategy sets \cite{dou2019finding}, but it is still a challenging research problem. 
	
	By considering the objective functions $\varphi_i(x)= \E_{\beta_x^i}(\theta_{i}(x,\cdot))$, we can derive the following direct theorem of existence of mixed equilibria for the Bayesian approach. The proof is similar to Theorem \ref{thm:ExistenceSingleLeader}, and the conclusion follows by applying Glicksberg Theorem \cite{Glicksberg1952}. 
	
	\begin{theorem}\label{thm:ExistenceMixedEq} Let $Y$ be a nonempty compact set of $\R^p$ and let us consider the bilevel game given by \eqref{eq:LeaderProblem}. Assume that this is a plain bilevel game, and consider the Bayesian approach \eqref{eq:LeaderProblem-Stochastic} with belief system $\{\beta^i\ : \ i\in \mathfrak{I}\}$ over $S:X \tto Y$. If for each leader $i\in \mathfrak{I}$ one has that
		\begin{enumerate}
			\item[(i)]the cost function $\theta_i$ is continuous,
			\item[(ii)] the global decision set $X_i$ is  nonempty and 
			compact,
			\item [(iii)] the belief $\beta^i:X\to\mathscr{P}(Y)$ is weak continuous,
		\end{enumerate}
		then, $x\mapsto \E_{\beta^i_x}(\theta_i(x,\cdot))$ is continuous for each $i\in \mathfrak{I}$, and so the bilevel game in the Bayesian approach admits a mixed equilibrium.
	\end{theorem}

	\begin{proof}
		Assuming that the bilevel game is plain, Problem \eqref{eq:LeaderProblem-Stochastic} can be rewritten as
		\begin{equation}\label{eq:LeaderProblem-Stochastic-reduction}
			P_i(x_{-i},\beta^i) = \begin{cases}
				\displaystyle\min_{x_i}\,\, &\E_{\beta^i_x}(\theta_i(x_i,x_{-i},y))\\
				\mbox{s.t.}\quad &x_i \in X_i.
			\end{cases}
		\end{equation}
		Here, each belief $\beta_i:X\to \mathscr{P}(Y)$ is well-defined due to the fact that in plain bilevel games we have that $X\subset \dom(S)$. Now, using the same argument as in the proof of Theorem \ref{thm:ExistenceSingleLeader} we can show that for each $i$ both $x\mapsto \E_{\beta^i}(\theta_i(x,\cdot))$ and $x\mapsto \E_{\beta^i}(-\theta_i(x,\cdot)) = -\E_{\beta^i}(\theta_i(x,\cdot))$ are both lower semicontinuous. Thus $x\mapsto \E_{\beta^i}(\theta_i(x,\cdot))$ is continuous and we conclude thanks to the Glicksberg Theorem \cite{Glicksberg1952}. 
	\end{proof}
	
	Note that in general, any selection $\beta:x\in X\mapsto \beta_x\in\mathscr{P}(S(x))$  constitutes a belief (for plain bilevel games; otherwise we replace $X$ by $D =\dom S$), and no other condition 
	is needed to define Problem \eqref{eq:LeaderProblem-Stochastic}. However, both existence results require the involved beliefs to be weak continuous. This is a strong property for a belief $\beta$ and, in particular, it entails that $\beta:X\times\mathcal{B}(Y)\to \R_+$ must be a Markov kernel (see, e.g., \cite[Definition 8.25]{klenke2013probability}) in the sense that\smallskip

	\begin{itemize}
		\item For each $x\in X$, $\beta(x,\cdot)$ is a Borel probability measure over $Y$; and
		\item For each $A\in\mathcal{B}(Y)$, $\beta(\cdot, A)$ is a $\mathcal{B}(X)$-measurable map.\smallskip
	\end{itemize}
	
	The first condition is trivially verified by definition, and the second one follows as a consequence of weak continuity: Indeed, it is enough to consider the maps $f = \ind_A$ with $A$ an open subset of $Y$ (which is a $\pi$-system of $\mathcal{B}(Y)$), and deduce from Theorem \ref{thm:Portemanteau} that the mappings $x\mapsto \beta(x,A)$ are lower semicontinuous, and therefore measurable (then, the desired conclusion follows from, e.g., \cite[Remark 2.6]{klenke2013probability}).
	While being a Markov kernel ensures several nice properties for a belief such as the measurability of $x\mapsto \E_{\beta_x}[\theta(x,\cdot)]$ and the well-posedness of mixed strategies (see Definition \ref{def:MixedEq}), it is not enough to guarantee existence of solutions for the Bayesian approach, as the following (classic) example shows.
	\begin{example}
		Consider $X=Y=[0,1]$ and $S:X\tto Y$ given by 
		\[
		S(x) = \begin{cases}
			\{0\}\qquad&\text{ if }x<1/2,\\
			[0,1]\qquad&\text{ if }x=1/2,\\
			\{1\}\qquad&\text{ if }x>1/2.
		\end{cases}
		\]
		Then, any belief $\beta:X\to \mathscr{P}(Y)$ over $S$ must have the form
		\[
		\beta(x) = \begin{cases}
			\delta_0\qquad&\text{ if }x<1/2,\\
			\mu\qquad&\text{ if }x=1/2,\\
			\delta_1\qquad&\text{ if }x>1/2.
		\end{cases}
		\]
		for some probability measure $\mu \in \mathscr{P}(Y)$, where $\delta_y \in \mathscr{P}(Y)$ stands for the delta measure at point $y\in Y$. In this case, all beliefs are Markov kernels, but none of them are weak continuous. 
	\end{example}

	\begin{remark}
		It is worth to note that the probability distributions obtained in a mixed strategy describe how leaders randomize their decision process. Following the two-stage interpretation of bilevel games, they describe how the random decision vector $x$ is produced in a first stage. In contrast, the beliefs model how the reaction $y\in S(x)$ might be produced in a second stage, as a function of a decision $x$ that has been already produced.
	\end{remark}
	
	\subsection{Rectangular continuity of set-valued maps}\label{subsec:RecContinuity}
	
	Most common probability distributions used in applications are derived from the uniform distribution (through a continuous density with respect to the Lebesgue measure). In our context, this core distribution is given by the Neutral belief (recall Definition \ref{def:neutralbelief}). Thus, we want to find sufficient conditions on the reaction map $S$, such that at least the Neutral belief is weak continuous. We will motivate the conditions with a simple example that shows usual continuity is not enough, by analyzing the centroid of the map.
	In fact, the conditions should guarantee that the centroid mapping $x\mapsto \centroid(S(x))$ is continuous, as the next lemma shows.
	
	\begin{lemma} Let $Y$ be a nonempty closed subset of $\R^p$, and $S:X\tto Y$ be a continuous set-valued map with nonempty convex compact values. If the Neutral belief $\iota:X\to \mathscr{P}(Y)$ over $S$ is weak continuous, then the centroid map $x\mapsto \centroid({S(x)})$ is continuous.
		\label{lem:centroidcontinuous}
	\end{lemma}
	\begin{proof}
		Fix $x\in X$. By definition, the centroid of $S(x)$ is given by the vector integral
		\[
		\centroid(S(x)) = 
		\sum_{k=1}^{p} \frac{1}{\lambda_{d_{x}}(S(x))}\left( \int_{S(x)} \langle e_k,y\rangle d\lambda_{d_{x}}(y) \right)e_k,
		\]
		where $e_k$ is the $k$th canonical vector of $\R^p$. Since for each $k\in \{1,\ldots,p\}$ the functional $\langle e_k,\cdot\rangle$ is continuous, then by Theorem \ref{thm:Portemanteau} the weak continuity of $\iota$ yields that the mapping $x\mapsto \centroid(S(x))$ is continuous. 
	\end{proof}
	
	A first guess of sufficient conditions for the weak continuity of the Neutral belief could be that the reaction map $S$ is continuous with convex values. However, in view of Lemma \ref{lem:centroidcontinuous}, the following example reveals that it is not so.
	
	\begin{example}\label{ex:triangletosegment}
		Let us consider the set-valued map $S:[0,1]\rightrightarrows [0,1]^2$ defined by 
		\[
		S(x)=\{(y_1,y_2)\in[0,1]^2: y_1\leq x\cdot y_2 \}
		\]
		which is clearly continuous and convex valued. We will call this map $S$ the triangle-to-segment map.
		Clearly, the centroid of $S(x)$ is $(\frac{x}{3},\frac{2}{3})$ for $x\in(0,1]$, while the centroid of $S(0)$ is $(0,\frac{1}{2})$. We see thus that the centroid mapping $x\mapsto \centroid(S(x))$ is not continuous at $x=0$. 
		In contrast, consider the map $\tilde{S}:\R\tto \R^2$ defined by $$\tilde{S}(x)=[0,1]\times [0,x]$$ which we call the rectangle-to-segment map. Clearly, $\tilde{S}$ satisfies that its centroid, given by $(1/2,x/2)$, is continuous. See Figure \ref{fig:triangletosement}.
	\end{example}
	\begin{figure}[h]
		\centering
		\begin{subfigure}{0.45\textwidth}
			\centering
			\begin{tikzpicture}[scale=0.8]
				\draw (0,0) -- (3,0) -- (3,1) --(0,0);
				\node at (3.7,.5) {\small $S(.4)$};
				\filldraw[black] (2,0.333) circle (1pt) node[anchor=west,above] {$\centroid$};
				\draw (0,0-1.5) -- (3,0-1.5) -- (3,.5-1.5) --     (0,0-1.5);
				\node at (3.7,.25-1.5) {$S(.2)$};
				\filldraw[black] (2,0.166-1.5) circle (1pt) node[anchor=west,above] {$\centroid$};
				\draw (0,0-3) -- (3,0-3);
				\node at (3.7,0-3) {$S(0)$};
				\filldraw[black] (1.5,0-3) circle (1pt) node[anchor=west,above right] {$\centroid$};
				\draw[dashed,gray] (2,.5) -- (2,0-3);
			\end{tikzpicture}
			\caption{Centroid of $S(x)$ is not continuous}
		\end{subfigure}
		\begin{subfigure}{0.45\textwidth}
			\centering
			\begin{tikzpicture}[scale=0.8]
				\draw (0,0) -- (3,0) -- (3,1) -- (0,1) --(0,0);
				\node at (3.7,.5) {\small $\tilde{S}(.4)$};
				\filldraw[black] (1.5,.5) circle (1pt) node[anchor=west,right] {$\centroid$};
				\draw (0,0-1.5) -- (3,0-1.5) -- (3,.5-1.5) --   (0,.5-1.5) --  (0,0-1.5);
				\node at (3.7,.5-1.75) {$\tilde{S}(.2)$};
				\filldraw[black] (1.5,0.25-1.5) circle (1pt) node[anchor=west,right] {$\centroid$};
				\draw (0,0-3) -- (3,0-3);
				\node at (3.7,0-3) {$\tilde{S}(0)$};
				\filldraw[black] (1.5,0-3) circle (1pt) node[anchor=west,above right] {$\centroid$};
				\draw[dashed,gray] (1.5,.5) -- (1.5,0-3);
			\end{tikzpicture}
			\caption{Centroid of $\tilde{S}(x)$ is continuous}
		\end{subfigure}
		
		\caption{Centroids of continuous set valued maps $S$ and $\tilde{S}$ (denoted by $\centroid$) in Example \ref{ex:triangletosegment} for different values of $x$. }
		\label{fig:triangletosement}
	\end{figure}
	

	The main problem with the triangle-to-segment map $S$ in the above example is that for all $x\in (0,1]$, the set $S(x)$ is not ``balanced", concentrating half of the ``mass" of the triangle in the third right part of it. In contrast, the limit set $S(0)$ is balanced in the sense that its ``mass" is equally distributed in both halves. Continuity of $S$ is not enough to prevent this abrupt change in ``mass" distribution. Motivated by the rectangle-to-segment map $\tilde{S}$, we introduce a stronger notion of continuity of set-valued maps, that ensures that changes of dimension may happen but in a ``balanced way''. As we will see in Section \ref{subsec:ConstructionBeliefs}, this notion is sufficient to ensure weak continuity not only for the Neutral belief, but also for a large family of beliefs built from continuous densities. The idea is to use an inner and an outer approximation of $S(x)$ by ``rectangles'' whose volume ratio goes to 1.
	
	
	\begin{definition}(Rectangular continuity) \label{def:rectangularcontinuity}
		Let $X\subset\R^d$ nonempty and $Y \subset \R^p$ nonempty and closed. Let  $S:X\tto Y$ be  a set-valued map with nonempty, bounded values. Let 
		$d_x = \adim (S(x))$ for each $x\in X$, and let $\bar{x}\in X$. We say that $S$ is \emph{rectangularly continuous} at $\bar{x}$ if there exist maps $T_0,T_1,R_0,R_1:X\tto Y$ such that 
		\begin{enumerate}
			\item For every $x$ near $\bar{x}$,
			\begin{equation}\label{eq:sandwich}
				T_0(x)+R_0(x)\subset S(x)\subset T_1(x)+R_1(x).
			\end{equation}
			\item For $j=0,1$, $R_j(x)\subset \spn(T_j(x)-T_j(x))^{\perp}\cap\spn(S(x)-S(x))$, and $T_j$ has constant affine dimension equal to $d_{\bar{x}}$.
			\item For $j=0,1$, 
			\begin{equation} \label{eq:limits}
				\Lim_{x\to\bar{x}} R_j(x)=\{0\} \text{ and }\Lim_{x\to\bar{x}}T_j(x)=S(\bar{x}).
			\end{equation}
			\item The maps $R_0$ and $R_1$ have measurable values and verify that 
			\begin{equation}\label{eq:ProportionOrthogonalPart}
				\lim_{x\to\bar{x}} \frac{\lambda_{d_x-d_{\bar{x}}}(R_0(x))}{\lambda_{d_x-d_{\bar{x}}}(R_1(x))}=1.
			\end{equation}
		\end{enumerate}
		We say that $S$ is \emph{rectangularly continuous} on $X$ if it is so at each point $\bar{x}\in X$.
	\end{definition}
	

	We propose the name of rectangular continuity since the inclusions in \eqref{eq:sandwich} enclose $S(x)$ between a rectangular-sum of sets: the ``horizontal parts" $T_0(x)$ and $T_1(x)$ converge to $S(\bar{x})$, while the ``vertical parts" $R_0(x)$ and $R_1(x)$ converge to $\{0\}$, but keeping the same proportion as stated by equation \eqref{eq:ProportionOrthogonalPart}. In the above definition, equation \eqref{eq:ProportionOrthogonalPart} is equivalent to
	\begin{equation*}
		\lim_{x\to\bar{x}} \frac{\lambda_{d_x}(R_0(x)+T_0(x))}{\lambda_{d_x}(R_1(x)+T_1(x))}=1.
	\end{equation*}

	
	Rectangular continuity at $\bar{x}\in X\subset\R^d$ implies continuity at $\bar{x}$. This follows directly from \eqref{eq:limits} and the Sandwich inclusion in \eqref{eq:sandwich}. The reciprocal implication is not always true (Example \ref{ex:triangletosegment} fails to verify rectangular continuity at $\bar{x}=0$). The following proposition provides some direct cases where rectangular continuity is verified.
	
	
	\begin{proposition}\label{prop:ContinuityWeakThanRectangular}
		Let $Y$ be a nonempty compact subset of $\R^p$, and let $S:X\tto Y$ be a set-valued map with nonempty closed values, that is continuous at $\bar{x}\in X$. Assume that one of the following conditions hold:
		\begin{enumerate}
			\item $S(\bar{x})$ is a singleton.
			\item $S$ is of constant affine dimension near $\bar{x}$.
			\item $S$ is convex-valued and $S(\bar{x})$ has nonempty interior.
			\item $S$ is convex-valued with $Y\subset \R$, that is, if $p=1$.
		\end{enumerate}
		Then, $S$ is rectangularly continuous at $\bar{x}$.
	\end{proposition}
	\begin{proof}
		The first assertion is verified by taking, for $j=0,1$, $T_j(x)$ as a single-valued selection of $S$ and $R_j(x)=S(x)-T_j(x)$. The second one comes from taking $T_j(x)=S(x)$ and $R_j(x)=\{0\}$, for $j=0,1$. The third one is a consequence of the second one, since in this case $S(x)$ has full affine dimension for all $x$ near $\bar{x}$. The fourth case is a consequence of the first and the third cases, since when $p=1$, the lower semicontinuity of $S$ yields either that $S(\bar{x})$ is a singleton, or that $S(x)$ has full affine dimension near $\bar{x}$. 
	\end{proof}
	

	It is important to observe that rectangular continuity really goes beyond the cases described in Proposition \ref{prop:ContinuityWeakThanRectangular} as we shall see in Section \ref{subsec:linearrectangular}. 
	The following example provides a nontrivial set-valued map $S$ which is rectangularly continuous precisely because the changes of affine dimension can be ``controled'' by rectangular decomposition.
	
	\begin{example} Let us consider the set-valued map $S:\R^n_+\tto \R^n_+$ given by
		\[
		S(x) = \left\{ y\in \R^n_+\, \Big | \begin{array}{c}
			y_i \leq x_i,\,\,\forall i=1,\ldots,n,  \\
			\displaystyle\sum y_i^2\leq \max_{j=1,\ldots,n} x_j^2.  
		\end{array}  \right\}.
		\]
		Clearly $\adim(S(x)) = |\mathrm{supp}(x)|$ where $\supp(x) = \{ i\ :\ x_i > 0  \}$ is the support of $x$. Thus, the affine dimension of $S(x)$ ranges from $0$ to $n$ having points where $\adim(S(x)) = k$ for any integer $k$ in between. Figure \ref{fig:exampleAllDimensions} illustrates $S$ when $n = 2$.

		\begin{figure}[h]
			\begin{center}
				\begin{tikzpicture} [x=1cm,y=1cm,scale=0.65]
					\coordinate (O1) at (0-2,0);
					
					\coordinate (X1) at (2-2,0);
					\coordinate (Y1) at (0-2,2);
					
					\draw[->] (O1) -- (X1);
					\draw[->] (O1) -- (Y1);
					\node at (1.5-2,1.5) {\scriptsize $S(2,1)$};
					\filldraw[fill=gray!30!white,line width=0.5mm] (O1) -- (1.5-2,0) arc[start angle=0, end angle=30,radius=1.5cm]--(0-2,0.75) -- (O1);
					\coordinate (Oex) at (4+2-4,0);
					\coordinate (Xex) at (6+2-4,0);
					\coordinate (Yex) at (4+2-4,2);
					
					\draw[->] (Oex) -- (Xex);
					\draw[->] (Oex) -- (Yex);
					\node at (4+1.5+2-4,1.5) {\scriptsize $S(2,2)$};
					\filldraw[fill=gray!30!white,line width=0.5mm] (Oex) -- (4+2-4,1.5) arc[start angle=90, end angle=0,radius=1.5cm] -- (Oex);
					
					\coordinate (O2) at (4+2,0);
					\coordinate (X2) at (6+2,0);
					\coordinate (Y2) at (4+2,2);
					
					\draw[->] (O2) -- (X2);
					\draw[->] (O2) -- (Y2);
					\node at (4+1.5+2,1.5) {\scriptsize $S(1,2)$};
					\filldraw[fill=gray!30!white,line width=0.5mm] (O2) -- (4+2,1.5) arc[start angle=90, end angle=60,radius=1.5cm]--(4+0.75+2,0) -- (O2);
					
					\coordinate (O3) at (2-4,-3);
					\coordinate (X3) at (4-4,-3);
					\coordinate (Y3) at (2-4,-1);
					
					\draw[->] (O3) -- (X3);
					\draw[->] (O3) -- (Y3);
					\node at (2-4+1,-2) {\scriptsize $S(2,0)$};
					\filldraw[fill=gray!30!white,line width=0.5mm] (O3) -- (2-4+1.5,-3);
					
					\coordinate (O4) at (2,-3);
					\coordinate (X4) at (4,-3);
					\coordinate (Y4) at (2,-1);
					
					\draw[->] (O4) -- (X4);
					\draw[->] (O4) -- (Y4);
					\node at (2+1,-2) {\scriptsize $S(0,2)$};
					\filldraw[fill=gray!30!white,line width=0.5mm] (O4) -- (2,-3+1.5);
					
					\coordinate (O5) at (2+4,-3);
					\coordinate (X5) at (4+4,-3);
					\coordinate (Y5) at (2+4,-1);
					
					\draw[->] (O5) -- (X5);
					\draw[->] (O5) -- (Y5);
					\node at (2+4+1,-2) {\scriptsize $S(0,0)$};
					\filldraw (O5) circle (0.1cm);
				\end{tikzpicture}
			\end{center}
			\caption{Values of $S(x)$ of different affine dimension when $n=2$.}
			\label{fig:exampleAllDimensions}
		\end{figure}
		Even though $S$ does not fit any of the cases of Proposition \ref{prop:ContinuityWeakThanRectangular}, we will see that it is rectangularly continuous as a particular case of Theorem \ref{thm:RecContinuityWeakAnalytic} in Section \ref{subsec:linearrectangular}.	
	\end{example}
	
	
	\subsection{Construction of weak continuous beliefs through densities}\label{subsec:ConstructionBeliefs}
	
	This section is devoted to show that rectangular continuity is sufficient to ensure weak continuity of certain beliefs that have a ``density" with respect to Lebesgue measures in the corresponding affine subspaces.
	
	\begin{definition}[Beliefs with density function]\label{def:ConditionalBelief} Let $Y$ be a nonempty closed subset of $\R^p$ and $S:X\tto Y$ be a set-valued map with nonempty closed values. A belief $\beta:X\to \mathscr{P}(Y)$ over $S$ is said to admit a density function with respect to the neutral belief (or simply to admit a density),
		if there exists a function $\rho: X\times Y\to [0,+\infty)$ such that
		\begin{enumerate}
			\item For all $x\in X$, $\rho(x,\cdot)$ restricted to $S(x)$ is $\lambda_{d_x}$-measurable with 
			\[
			0<\int_{S(x)}\rho(x,y)d\lambda_{d_x}(y) <+\infty;
			\]
			
			\item And, for all $A\in \mathcal{B}(Y)$
			\begin{equation}\label{eq:DefConditionalBelief}
				\beta_x(A) = \frac{\int_{A\cap S(x)} \rho(x,y) d\lambda_{d_x}(y)}{\int_{S(x)} \rho(x,y) d\lambda_{d_x}(y)},
			\end{equation}
		\end{enumerate}
		where $d_x = \adim(S(x))$. In such a case, we call $\rho$ a density function of $\beta$.
	\end{definition}
	
	The name \emph{density function} is inspired in Markov kernels with densities (see, e.g., \cite[Definition 1.2.4]{Bakry2013Markov}), which are kernels $\beta:X\times Y\to \R_+$ where each measure $\beta(x,\cdot)$ can be written as
	\[
	\beta(x,A) = \int_A \rho(x,y)dm(y),
	\]
	for some fixed measure $m\in \mathscr{P}(Y)$. However, our definition is a little different, since the density function is integrated with respect to the Lebesgue measure $\lambda_{d_x}$ associated to the moving set $S(x)$, instead of the (fixed) Lebesgue measure $\lambda$ of $Y$. Note that the density function of a belief is not uniquely defined. However, if a belief $\beta$ admits a density function $\rho$, then 
	\begin{equation}\label{eq:RadonNikodym-kernel}
		\frac{1}{\int_{S(x)}\rho(x,y)d\lambda_{x}(y)}\rho(x,\cdot) = \frac{\partial \beta_x}{\partial \lambda_x},
	\end{equation}
	where $\frac{\partial \beta_x}{\partial \lambda_{d_x}}$ is the Radon-Nikodym derivative of $\beta_x$ with respect to $\lambda_{d_x}$, over the measurable space $(S(x),\mathcal{B}(X))$. It is worth noting that one could consider densities with respect to other beliefs, just by replacing the role of the Neutral belief in Definition \ref{def:ConditionalBelief}.
	
	
	We start our development with the following lemma that links the continuity of a set-valued map $S$ with the continuity of a measure evaluated in the images of $S$.
	
	\begin{lemma}\label{lemma:continuityofmesure}
		Let $Y$ be a nonempty closed subset of $\R^p$ and let $\mu$ be a finite Borel measure on $Y$. Let $S:X\tto Y$ be a set-valued map with measurable values and let $\bar{x}\in X$. The following assertions hold: 
		\begin{enumerate}
			\item If $S$ is upper semicontinuous at $\bar{x}$, then $\mu(S(\cdot))$ is upper semicontinuous at $\bar{x}$.
			\item If $S$ is lower semicontinuous at $\bar{x}$, with compact convex values near $\bar{x}$, and $\mu(\partial S(\bar{x}))=0$, then $\mu(S(\cdot))$ is lower semicontinuous at $\bar{x}$. 
		\end{enumerate}
	\end{lemma}
	\begin{proof}
		1. Consider the sequence of sets $E_n:=S(\bar{x})+\tfrac{1}{n}\Ball$ decreases to $S(\bar{x})$, that is, $E_{n+1}\subset E_n$ for all $n\in \N$ and $\bigcap_n E_n=S(\bar{x})$. Then, by continuity of the measure, there exists $n_0\in\N$ such that $\mu(S(\bar{x}))+\varepsilon>\mu(E_{n_0})$. Now since $S$ is upper semicontinuous, there exists $\delta>0$ such that $ S(x)\subset S(\bar{x})+\tfrac{1}{n_0}\Ball= E_{n_0}$ for any $x\in B(\bar{x},\delta)\cap X$, and so $\mu(S(x))\leq\mu(E_{n_0})$. Putting both inequalities together, we have 
		\[ 
		\mu (S(x))<\mu(S(\bar{x}))+\varepsilon,\quad \forall x \in B(\bar{x},\delta)\cap X. 
		\]
		Since the above argument works for any $\varepsilon>0$, we have proved that $\mu(S(\cdot))$ is upper semicontinuous at $\bar{x}$.
		
		2. If $\into(S(\bar{x})) = \emptyset$, then $\mu(S(\bar{x})) = \mu(\partial S(\bar{x})) = 0$. Thus, the lower semicontinuity of $\mu(S(\cdot))$ at $\bar{x}$ follows trivially from the non-negativity of $\mu$. Assume then that $\into(S(\bar{x}))$ is nonempty and let us consider the sequence of sets $(K_n)_{n\in\N}$ defined by
		\[
		K_n:=\left \{x: B\left (x,\tfrac{1}{n}\right )\subset S(\bar{x}) \right \}.
		\]
		Clearly, $K_n$ increases to $ \into(S(\bar{x}))$, that is, $\bigcup_{n}K_n=\into(S(\bar{x}))$ and $K_n\subset K_{n+1}$ for all $n\in \N$. Thus, by the continuity of the measure, given $\varepsilon>0$ there exists $n_1$ such that 
		\begin{equation}
			\mu(S(\bar{x}))=\mu(\into( S(\bar{x})))<\mu (K_{n_1})+\varepsilon,
			\label{eq:aprox1}
		\end{equation}
		where the first equality is due to the assumption $\mu(\partial S(\bar{x}))=0$. Now, from the lower semicontinuity of $S$ at $\bar{x}$ (see \cite[Proposition 5.12]{rockafellar2009variational}) we know there exists $\delta>0$ such that 
		\[ 
		S(\bar{x})\subset S(x)+\tfrac{1}{n_1}\Ball,\quad \forall x\in B(\bar{x},\delta)\cap X.
		\]
		Since $K_{n_1}+\tfrac{1}{n_1}\Ball \subset S(\bar{x})$, we get that for any $x\in B(\bar{x},\delta)\cap X$ one has 
		\[ 
		K_{n_1}+\tfrac{1}{n_1}\Ball\subset S(x)+\tfrac{1}{n_1}\Ball.
		\]
		Using the cancellation law for convex sets (see \cite[Corollary 3.35]{rockafellar2009variational}), we obtain that for every $x\in B(\bar{x},\delta)\cap X$, $ K_{n_1}\subset S(x)$ and so $\mu(K_{n_1})\leq \mu(S(x))$. Combining this last inequality with \eqref{eq:aprox1} we get
		\[
		\mu(S(\bar{x}))\leq\mu(S(x))+\varepsilon,\quad \forall x\in B(\bar{x},\delta)\cap X, 
		\]
		which proves the lower semicontinuity of $\mu(S(\cdot))$ at $\bar{x}$. 
	\end{proof}
	
	For a solution map $S$ whose values are full-dimensional, we can state a first theorem that provides sufficient conditions to ensure weak continuity of beliefs admitting density functions, in the spirit of \cite[Theorem 3.1]{Mallozzi1996}. Here, we write $L^1:= L^1(Y,\mathcal{B}(Y),\lambda_k)$, as the usual space of Lebesgue-integrable functions over $Y$, with $k=\adim(Y)$.
	
	\begin{theorem}\label{thm:WeakContinuityFullDim}
		Let $Y$ be a nonempty convex compact subset of $\R^p$, $S:X\tto Y$ be a set-valued map with nonempty convex and compact values, and $\bar{x}\in X$. Suppose that $S$ is continuous at $\bar{x}$ and that the relative interior of $S(x)$ with respect to $Y$ is nonempty, for each $x$ near $\bar{x}$. Let $\beta:X\to \mathscr{P}(Y)$ be a belief over $S$ with a density function $\rho$ that satisfies that $x\mapsto \rho(x,\cdot)$ is $L^1$-continuous at $\bar{x}$, that is,
		\[
		\| \rho(x,\cdot) - \rho(\bar{x},\cdot) \|_{L^1} \xrightarrow{x\to \bar{x}} 0.
		\]
		Then, $\beta$ is weak continuous at $\bar{x}$.
	\end{theorem}
	
	\begin{proof}
		Without losing any generality, we assume that $Y$ has nonempty interior. Define the measure $\mu$ given by
		\[
		\mu(A) = \int_A \rho(\bar{x},y)d\lambda(y).
		\]
		By Lemma \ref{lemma:continuityofmesure}, the mapping $x\mapsto \mu(S(x))$ is continuous at $\bar{x}$ and, for any closed subset $C\subset Y$ intersecting $S(\bar{x})$, the mapping $x\mapsto\mu(S(x)\cap C)$ is upper semicontinuous at $\bar{x}$.
		Let $C$ be a closed subset of $Y$. By the Portemanteau Theorem, it is enough to prove that
		\[
		\limsup_{x\to \bar{x}} \beta_x(C) \leq \beta_{\bar{x}}(C).
		\]
		Such inequality holds trivially if $C\cap S(\bar{x})=\emptyset$, due to the upper semicontinuity of $S$. Thus, we assume that $C\cap S(\bar{x})\neq\emptyset$. Now, for $x$ near enough to $\bar{x}$ we have that $\mu(S(x)) - \|\rho(x,\cdot) - \rho(\bar{x},\cdot)\|_{L^1}$ is strictly positive, and so we can write
		\begin{align*}
			\limsup_{x\to \bar{x}}\beta_x(C) &= \limsup_{x\to \bar{x}}\frac{\int_{S(x)\cap C}\rho(x,y)d\lambda(y)}{\int_{S(x)}\rho(x,y)d\lambda(y)}\\
			&\leq \limsup_{x\to \bar{x}}\frac{\int_{S(x)\cap C}\rho(x,y)d\lambda(y)}{\mu(S(x)) - \|\rho(x,\cdot) - \rho(\bar{x},\cdot)\|_{L^1}}\\
			&\leq \limsup_{x\to \bar{x}}\frac{\mu(S(x)\cap C) + \|\rho(x,\cdot) - \rho(\bar{x},\cdot)\|_{L^1}}{\mu(S(x)) - \|\rho(x,\cdot) - \rho(\bar{x},\cdot)\|_{L^1}}\\
			& \leq \frac{\mu(S(\bar{x})\cap C)}{\mu(S(\bar{x}))} = \beta_{\bar{x}}(C).
		\end{align*}
		The proof is then finished. 
	\end{proof}
	
	
	By invoking Theorem \ref{thm:ExistenceSingleLeader}, the above result yields the following direct corollary, that recovers the existence result of \cite[Theorem 3.2]{Mallozzi1996} for regularized bilevel programs with continuous data.
	
	\begin{corollary}\label{cor:ExistenceRegularized} Let $Y$ be a nonempty convex compact subset of $\R^p$. Let us consider a bilevel game given by \eqref{eq:LeaderProblem-Stochastic}  with only one leader (i.e., $|\mathfrak{I}|=1$), for which $S$ is given as the $\varepsilon$-argmin of a lower-level problem for some fixed $\varepsilon>0$, that is,
		\begin{equation}\label{eq:RegularizedFollowerProblem}
			S(x) = \left\{ y\in \R^p\, :\ f(x,y)\leq \min_{y\in Y} f(x,y) + \varepsilon,\, y\in Y \right\}.
		\end{equation}
		If $\theta$ and $f$ are continuous and $X$ is nonempty compact, then the Bayesian approach 
		admits a solution for any belief $\beta$ over $S$ with 
		a density function $\rho$ satisfying that $x\mapsto\rho(x,\cdot)$ is $L^1$-continuous.
	\end{corollary}
	
	
	To continue, we study set-valued maps with constant affine dimension, but not necessarily full-dimensional. This condition is a particular case of rectangular continuity, as stated in Proposition \ref{prop:ContinuityWeakThanRectangular}.
	
	\begin{lemma}\label{lem:continuityofkvolume}
		Let $Y$ be a nonempty closed subset of $\R^p$ and let $\bar{x}\in X$. Let $T:X\tto Y$ be a  set-valued map with nonempty convex and compact values, and with constant affine dimension $k$, which is continuous at $\bar{x}$. For any continuous function $\rho: X\times Y\to [0,+\infty)$, the decision-dependent measure $x\mapsto \mu_x$ given by
		\[
		\forall A\in\mathcal{B}(Y),\quad \mu_x(A) = \int_{A\cap T(x)} \rho(x,y)d\lambda_k(y),
		\]
		satisfies that the mapping $x\mapsto \mu_x(T(x))$ is continuous at $\bar{x}$.
	\end{lemma}
	
	\begin{proof}
		We divide the proof in two steps: First we will show the result for $\rho(x,y)\equiv 1$, that is, for $\mu_x \equiv \lambda_k$; and secondly for the general case.
		
		\textbf{Step 1}: Define $D(x)=T(x)+R(x)$, where $R(x)=\Ball\cap L(x)^\perp$ and $L(x) = \spn(T(x)-T(x))$. The map $D$ is clearly  nonempty and convex valued and moreover, $D$ must be continuous at $\bar{x}$.
		
		To prove the continuity of $D$ at $\bar{x}$, we will first show that $R$ is continuous at $\bar{x}$. To do so, we will construct single-valued mappings  $b_i:X\to\Ball$ with $i\in \{1,\ldots,k\}$ such that all are continuous at $\bar{x}$ and that $\{b_i(x)\ :\ i\in\{1,\ldots,k\}\}$ is an orthogonal unit basis of $L(x)$ for each $x$ near enough to $\bar{x}$.
		
		Since $T$ is lower semicontinuous at $\bar{x}$, 
		it is easy to verify that there exists a selection $s:X\to \R^p$ of $T$ which is continuous at $\bar{x}$. Replacing $T(x)$ by $T(x)-s(x)$, we may and do assume that $0\in T(x)$ for each $x\in X$, and thus $L(x) = \spn(T(x))$. Clearly, $\dim(L(x)) = k$ for each $x\in X$.
		
		Now, let $\{\bar{y}_1,\ldots,\bar{y}_k\}\subset T(\bar{x})$ be a basis of $L(\bar{x})$. For each $i\in \{1,\ldots,k\}$, we can define the set-valued map $F_i:X\tto Y$ given by
		\[
		F_i(x) = \begin{cases}
			T(x)\qquad&\mbox{ if }x\neq\bar{x},\\
			\{\bar{y}_i\}&\mbox{ if }x = \bar{x}.
		\end{cases}
		\]
		It is not hard to verify that for each $i\in \{1,\ldots,k\}$, $F_i$ remains lower semicontinuous at $\bar{x}$, and so we know that there exists a selection $s_i:X\to \R^p$ of $F_i$ which is continuous at $\bar{x}$. Clearly, the mappings $s_i$ are also selections of $T$ and, by construction, $\{ s_i(\bar{x})\ :\ i\in \{1,\ldots,k\} \}$ is linearly independent. It is known that (see, e.g., \cite[Chapter 5, Exercise 43a]{pugh2002real}) the continuity of the selections at $\bar{x}$ ensures that there is a neighborhood $U$ of $\bar{x}$ such that
		\begin{equation}\label{eq:inProof_linearIndepNeighborhood}
			\{ s_i(x)\ :\ i\in \{1,\ldots,k\} \}\text{ is linearly independent, for each }x\in U\cap X.
		\end{equation}
		
			
		Without losing generality, let us assume that $U = \R^d$. Now, applying the classic Gram-Schmidt algorithm, we can define the mappings $b_i:X\to \R^p$ for each $i\in \{1,\ldots,k\}$ where $\{b_i(x)\ :\ i=1,\ldots k\}$ is an orthogonal unit basis of $L(x)$. 
		It is not hard to see that the mappings $b_i$ remain continuous at $\bar{x}$, and so the desired construction is then completed.
		
		Now, to prove that $R$ is continuous at $\bar{x}$, it is enough to show that
		\[
		\Limsup_{x\to \bar{x}} R(x) \subset R(\bar{x})\subset \Liminf_{x\to \bar{x}} R(x).
		\]
		
		For the first inclusion, choose a sequence $(x_n)\subset X$ converging to $\bar{x}$ and a sequence $(y_n)\subset \R^p$ converging to $y\in \R^p$, satisfying that $y_n\in R(x_n)$ for each $n\in \N$. We need to show that $y\in R(\bar{x})$.
		
		Firstly, since $y_n\in \Ball$ for each $n\in \N$, it is clear that $y\in \Ball$. Now, on the one hand, since $y_n\in L(x_n)^{\perp}$, we have that $\proj(y_n; L(x_n)) = 0$.
		On the other hand, for $n\in\N$, we can write
		\[
		\proj(y_n; L(x_n)) = \sum_{i=1}^k \langle y_n,b_i(x_n)\rangle b_i(x_n),
		\]
		which yields, by continuity  at $\bar{x}$ of the mappings $b_i$, that 
		\[
		\proj(y_n; L(x_n)) \xrightarrow{n\to\infty} \sum_{i=1}^k \langle y,b_i(\bar{x})\rangle b_i(\bar{x}) = \proj(y;L(\bar{x})).
		\]
		We deduce that $\proj(y;L(\bar{x})) = 0$, proving that $y\in L(\bar{x})^{\perp}$ and so $y\in R(\bar{x})$.
		
		Now, for the second inclusion, choose a a sequence $(x_n)\subset X$ converging to $\bar{x}$ and a point $y\in R(\bar{x})$. We need to show that there exists a sequence $(y_n)$ converging to $y$ and satisfying that $y_n\in R(x_n)$ for all $n\in\N$ large  enough.
		
		Consider then, the sequence $(y_n)$ defined by $y_n = y - \proj(y;L(x_n))$. Since $y\in \Ball$, we get that $y_n\in \Ball$ for each $n\in \N$. Furthermore, by construction $y_n\in L(x_n)^{\perp}$ and so $y_n\in R(x_n)$ for each $n\in \N$. As we did before,
		\[
		y_n = y - \sum_{i=1}^k \langle y,b_i(x_n)\rangle b_i(x_n)\xrightarrow{n\to \infty} y - \sum_{i=1}^k \langle y,b_i(\bar{x})\rangle b_i(\bar{x}) = y - \proj(y;L(\bar{x})).
		\]
		The proof is finished by noting that, since $y\in L(\bar{x})^{\perp}$, then $y - \proj(y;L(\bar{x})) = y$. We have shown then that the set-valued map $R:X\tto \R^p$ is continuous at $\bar{x}$.
		Now, $D$ is the sum  of two set-valued maps that are continuous at $\bar{x}$, one of them being locally bounded, then $D$ is also continuous at $\bar{x}$ (see e.g. Exercise 5.24 in \cite{rockafellar2009variational}). To finish Step 1, we observe that
		\[
		\lambda(D(x))=\lambda_k(T(x))\lambda_{p-k}(R(x)),
		\]
		where $\lim_{x\to\bar{x}}\lambda(D(x))=\lambda(D(\bar{x}))$ (Lemma \ref{lemma:continuityofmesure}) and $\lambda_{p-k}(R(x))$ is a positive constant.
		Form this it directly follows that $x\mapsto\lambda_k(T(x))$ is continuous at $\bar{x}$.
		
		\textbf{Step 2:} Let $(x_n)\subset X$ be any sequence converging to $\bar{x}$. We need to show that
		\[
		\int_{T(x_n)}\rho(x_n,y)d\lambda_k(y) \to \int_{T(\bar{x})}\rho(\bar{x},y)d\lambda_k(y). 
		\]
		By upper semicontinuity of $T$, we can assume that $T(x_n)\subset T(\bar{x})+\Ball$. Fix $\varepsilon>0$. Let $U$ be a open neighborhood of $\bar{x}$ in $X$ such that $K=\sup_{U\times Y}\rho < +\infty$. Such a neighborhood exists by continuity of $x\mapsto \sup_{Y}\rho(x,\cdot)$ (see, e.g. \cite[Theorem 2.3.1]{Ichiishi1983}). By shrinking $U$ if necessary, there exists $\{ V_i\ :i\in I \}$ a finite measurable partition of $T(\bar{x})+\Ball$ such that $\diam(\rho(U\times V_i)) <\varepsilon$ for all $i\in I$. Then, for $n\in\N$ large enough (such that $x_n\in U$) and for every $i\in I$, we can write
		\begin{align*}
			&\int_{T(x_n)\cap V_i}\rho(x_n,y)d\lambda_k(y) -\int_{T(\bar{x})\cap V_i}\rho(\bar{x},y)d\lambda_k(y)\\
			\leq& (\sup_{U\times V_i}\rho) (\lambda_k(T(x_n)\cap V_i) - \lambda_k(T(\bar{x})\cap V_i)) + \diam(\rho(U\times V_i))\lambda_k(T(\bar{x})\cap V_i)\\
			&\leq K(\lambda_k(T(x_n)\cap V_i) - \lambda_k(T(\bar{x})\cap V_i)) + \varepsilon\lambda_k(T(\bar{x})\cap V_i).
		\end{align*}
		By summing the above inequalities through $i\in I$, we get that
		\[
		\int_{T(x_n)}\rho(x_n,y)d\lambda_k(y) -\int_{T(\bar{x})}\rho(\bar{x},y)d\lambda_k(y) \leq K(\lambda_k(T(x_n)) - \lambda_k(T(\bar{x}))) + \varepsilon\lambda_k(T(\bar{x})).
		\]
		Using a similar argument, we can show that
		\begin{align*}
			\left|\int_{T(x_n)}\rho(x_n,y)d\lambda_k(y) -\int_{T(\bar{x})}\rho(\bar{x},y)d\lambda_k(y)\right| &\leq K|\lambda_k(T(x_n)) - \lambda_k(T(\bar{x}))| + \varepsilon\lambda_k(T(\bar{x}))\\
			&\xrightarrow{n\to\infty} \varepsilon\lambda_k(T(\bar{x})),
		\end{align*}
		where the convergence follows from Step 1 of this proof. Since $\varepsilon$ is arbitrary, the proof is finished. 
	\end{proof}

	Now, we are ready to state and prove the main result of this section.
	
	\begin{theorem}\label{thm:WeakContinuityUnderRecContinuity}
		Let $Y$ be a nonempty compact subset of $\R^p$, $S:X\tto Y$ be a set-valued map with nonempty convex values, and let $\bar{x}\in X$. Let $\beta:X\to \mathscr{P}(Y)$ be a belief over $S$ with continuous density function $\rho$. If $S$ is rectangularly continuous at $\bar{x}$, then $\beta$ is weak continuous at $\bar{x}$. In particular, the Neutral belief $\iota:X\to\mathscr{P}(Y)$ over $S$ is weak continuous under rectangular continuity of $S$.
	\end{theorem}
	
	\begin{proof}
		Let us consider an arbitrary sequence $x_n\to\bar{x}$.
		The theorem is reduced to prove that 
		$$\beta_{x_n}\xrightarrow{w}\beta_{\bar{x}}.$$ We divide the proof of this into two steps: In Step 1 we consider the case when $S$ has constant affine dimension, while in Step 2 we consider the general case of a rectangularly continuous $S$. In the latter case, from the definition of rectangular continuity, $S$ is enclosed between $T_0+R_0$ and $T_1+R_1$, where $T_0$ and $T_1$ have both constant affine dimension. Thus, we exploit what was proved in Step 1.
		
		\textbf{Step 1:} Let us assume here that $S$ has constant affine dimension $\adim(S(x)) = k$ for all $x\in X$. From the Portemanteau Theorem (see Theorem \ref{thm:Portemanteau}) it is enough to prove that for any closed set $C\subset Y$ we have
		\begin{equation}
			\limsup_n \beta_{x_n}(C)\leq \beta_{\bar{x}}(C).
			\label{ineq:portemanteauCerrado}
		\end{equation}
		Define the set-valued map $D:X\tto \R^p$ given by $D(x) = (S(x)\cap C)\times R(x)$, 
		where $R(x) = \Ball\cap \spn(S(x) - S(x))^{\perp}$. If $D(\bar{x}) = \emptyset$, then the upper semicontinuity of $S$ at $\bar{x}$ will entail that $D(x) = \emptyset$ near $\bar{x}$, and inequality \eqref{ineq:portemanteauCerrado} would follow trivially. Thus, let us assume that $D(\bar{x})\neq\emptyset$. Using the same arguments as in the first step of the proof of Lemma \ref{lem:continuityofkvolume}, we know that $D$ is upper semicontinuous. Thus, by Lemma \ref{lemma:continuityofmesure},
		\[
		\limsup_{n} \lambda(D(x_n)) \leq\lambda (D(\bar{x})), 
		\]
		which entails that $x\mapsto \lambda_k(S(x)\cap C)$ is upper semicontinuous at $\bar{x}$. Then, we can easily adapt the same development of the second step of the proof of Lemma \ref{lem:continuityofkvolume}, to conclude that
		\[
		\limsup_n \int_{S(x_n)\cap C} \rho(x_n,y)d_{\lambda_k}(y) \leq \int_{S(\bar{x})\cap C} \rho(\bar{x},y)d_{\lambda_k}(y).
		\]
		Finally, using Lemma \ref{lem:continuityofkvolume}, we can write
		\begin{align*}
			\limsup_n\beta_{x_n}(C) &= \limsup_n\frac{  \int_{S(x_n)\cap C} \rho(x_n,y)d\lambda_{k}(y)}{\int_{S(x_n)} \rho(x_n,y)d\lambda_{k}(y)}\\
			&\leq \frac{\int_{S(\bar{x})\cap C} \rho(\bar{x},y)d\lambda_{k}(y)}{\int_{S(\bar{x})} \rho(\bar{x},y)d\lambda_{k}(y)} = \beta_{\bar{x}}(C).
		\end{align*}
		The proof of the first part is finished.
		
		\textbf{Step 2: Now, we consider the general case.} From the Portemanteau Theorem (see Theorem  \ref{thm:Portemanteau}) it is enough to prove that for any open set $U$ in $Y$ we have
		\begin{equation}
			\liminf_n \beta_{x_n}(U)\geq \beta_{\bar{x}}(U).
			\label{ineq:portemanteauA}
		\end{equation}
		Note that $x\mapsto \adim(S(x))$ is lower semicontinuous, because $S$ is continuous with convex values. Furthermore, since it ranges over finitely many values, we can assume without loss of generality that $\adim (S(x_n))=l$ for all $n\in\N$, and $\adim(S(\bar{x}))=k\leq l$. To simplify notation, for $r \in \{0,1,\ldots,l\}$, we define the Borel measures $\mu_n^r$ and $\mu^r$ as
		\[
		\mu_n^r(A) = \int_{A\cap S(x_n)} \rho(x_n,y)d\lambda_r(y)\quad\text{ and }\quad\mu^r(A) = \int_{A\cap S(\bar{x})} \rho(\bar{x},y)d\lambda_r(y).
		\]
		
		Let $T_0,T_1,R_0,R_1:X\tto Y$ be the set-valued maps given by the definition of rectangular continuity (Definition \ref{def:rectangularcontinuity}), so that in particular
		\begin{equation*}
			T_0(x_n)+R_0(x_n)\subset S(x_n)\subset T_1(x_n)+R_1(x_n).
		\end{equation*}
		
		We claim that, for any given $\varepsilon>0$  there exists $n_0$ large enough such that for every $n\geq n_0$ we can approximate $\mu_n^l(S(x_n))$ from below as
		\begin{equation}
			(\mu_{n}^k(T_0(x_n))-\varepsilon)\lambda_{l-k}(R_0(x_n))\leq \mu_{n}^l(S(x_n)).
			\label{eq:firstineq}
		\end{equation}
		
		Set $\varepsilon'=\frac{\varepsilon}{2\lambda_k(S(\bar{x}))}>0$. Replacing $X$ by $\{x_n\ :\ n\in\N\}\cup\{\bar{x}\}$, we can assume without losing generality that $X$ is compact. 
		Thus, since $X$ and $Y$ are compact, $\rho$ is uniformly continuous and so there exists $\delta>0$ such that 
		\[
		\rho(x_n,u+v)>\rho(x_n,u)-\varepsilon',\quad \forall v\in \delta\Ball, \:\forall u\in Y, \forall n\geq n_0,
		\]
		for some $n_0\in\N$. Since $\lambda_k(T_0(x_n))\to\lambda_k(T_0(\bar{x}))$ (see Lemma \ref{lem:continuityofkvolume}) and $\lim R_0(x_n)=\{0\}$, 
		then for $n$ large enough we have that $\varepsilon>\varepsilon'\lambda_k(T_0(x_n))$ and $R_0(x_n)\subset \delta\Ball$. So, by $\lambda_k$-integrating over $T_0(x_n)$ we obtain 
		\begin{align*}
			\int_{T_0(x_n)}\rho(x_n,u+v)d\lambda_{k}(u)
			&\geq\int_{T_0(x_n)}(\rho(x_n,u)-\varepsilon' )d\lambda_k(u)\\
			&=\mu_{n}^k(T_0(x_n))-\varepsilon'\lambda_k(T_0(x_n))\\
			&\geq\mu_{n}^k(T_0(x_n))-\varepsilon.
		\end{align*}
		Therefore, by a mild application of the Fubini Theorem, if we now $\lambda_{l-k}$-integrate the previous inequality over $R_0(x_n)$ we obtain 
		\begin{align*}
			\mu_{n}^l(S(x_n))&\geq \mu_{n}^l(T_0(x_n)+R_0(x_n))\\
			&\geq(\mu_{n}^k(T_0(x_n))-\varepsilon)\lambda_{l-k}(R_0(x_n)).
		\end{align*}
		The claim stated in equation \eqref{eq:firstineq} is then proven. Following similar arguments as for \eqref{eq:firstineq}, for every $n$ large enough we can also approximate $\mu_n^l(S(x_n))$ from above as
		\begin{equation}
			\mu_{n}^l(S(x_n))\leq(\mu_{n}^k(T_1(x_n))+\varepsilon)\lambda_{l-k}(R_1(x_n))
			\label{eq:secondineq}
		\end{equation}
		
		Thus, by combining inequalities \eqref{eq:firstineq} and \eqref{eq:secondineq}, we can write for every $\varepsilon>0$ and $n\in\N$ large enough
		\[
		(\mu_{n}^k(T_0(x_n))-\varepsilon)\leq \frac{\mu_{n}^l(S(x_n))}{\lambda_{l-k}(R_0(x_n))}\leq(\mu_{n}^k(T_1(x_n))+\varepsilon)\frac{\lambda_{l-k}(R_1(x_n))}{\lambda_{l-k}(R_0(x_n))}.
		\]
		
		Recalling from \eqref{eq:ProportionOrthogonalPart} that $\lambda_{l-k}(R_1(x_n))/\lambda_{l-k}(R_0(x_n))\to 1$, noting that from Lemma \ref{lem:continuityofkvolume} we have
		$\lim\mu_{n}^k(T_0(x_n))=\lim \mu^k(T_1(\bar{x}))=\mu^k(S(\bar{x}))$, and taking $\varepsilon\to 0$, we deduce that 
		\begin{equation}
			\lim_n \frac{\mu_{n}^l(S(x_n))}{\lambda_{l-k}(R_0(x_n))}=\mu^k(S(\bar{x})).
			\label{eq:fracmeasures}
		\end{equation}
		Recall that we aim to prove \eqref{ineq:portemanteauA}, which multiplied by \eqref{eq:fracmeasures} simplifies to
		\begin{equation}
			\liminf_n \frac{\mu_{n}^l(S(x_n)\cap U)}{\lambda_{l-k}(R_0(x_n))}\geq \mu^k(S(\bar{x})\cap U).
			\label{eq:toprove}
		\end{equation}
		So, in order to prove \eqref{eq:toprove}, let us fix $\varepsilon>0$ and define $U_\varepsilon := \{y\in Y\ : B(y,\varepsilon)\cap Y \subset U\}$. Clearly, $U_{\varepsilon}\subset U$ and, if $n$ is large enough, we have $R_{0}(x_n)\subset \varepsilon \Ball$. Then $U_\varepsilon + R_0(x_n)\subset U$ and therefore
		\begin{align*}
			(T_0(x_n)\cap U_\varepsilon)+R_0(x_n)&\subset (T_0(x_n)+ R_0(x_n)) \cap (U_\varepsilon +R_0(x_n))\\
			&\subset S(x_n)\cap U.
		\end{align*}
		Again, following similar arguments as for the proof of \eqref{eq:firstineq},
		we obtain that 
		\begin{align*}
			(\mu_{n}^k(T_0(x_n)\cap U_\varepsilon)-\varepsilon)\lambda_{l-k}(R_0(x_n)) &\leq \mu_{n}^l(T_0(x_n)\cap U_\varepsilon + R_0(x_n))\\ &\leq \mu_{n}^l(S(x_n)\cap U).
		\end{align*}
		Therefore
		\[
		\liminf_n\frac{\mu_{n}^l(S(x_n)\cap U)}{\lambda_{l-k}(R_0(x_n))}\geq \liminf_n \mu_{n}^k(T_0(x_n)\cap U_\varepsilon)-\varepsilon\geq \mu^k(T_0(\bar{x})\cap U_\varepsilon)-\varepsilon,
		\]
		where the last inequality follows from Lemma \ref{lem:continuityofkvolume} and Step 1 of this proof, applied to $T_0$ instead of $S$. Indeed,
		\begin{align*}
			\liminf_n \mu_{n}^k(T_0(x_n)\cap U_\varepsilon) &= \liminf_n \frac{\mu_{n}^k(T_0(x_n)\cap U_\varepsilon)}{\mu_n^k(T_0(x_n))}\cdot\mu_n^k(T_0(x_n))\\
			&\geq \liminf_n \frac{\mu_{n}^k(T_0(x_n)\cap U_\varepsilon)}{\mu_n^k(T_0(x_n))}\cdot\liminf_n\mu_n^k(T_0(x_n))\\
			&= \frac{\mu^k(T_0(\bar{x})\cap U_\varepsilon)}{\mu^k(T_0(\bar{x}))}\cdot\mu^k(T_0(\bar{x})) = \mu^k(T_0(\bar{x})\cap U_\varepsilon).
		\end{align*} 	
		Finally, since $\mu^k(T_0(\bar{x})\cap U_\varepsilon)$  increases to $\mu^k(T_0(\bar{x})\cap U)=\mu^k(S(\bar{x})\cap U)$ as $\varepsilon$ goes to $0$, then we conclude that \eqref{eq:toprove} holds, finishing the proof. 
	\end{proof} 
	
	

	\subsection{Linear parametric problems verify rectangular continuity}\label{subsec:linearrectangular} 
	
	In this subsection we assume that the reaction map $S$ corresponds to the solution of a linear parametric optimization problem of the form
	
	\begin{equation}\label{eq:LinearFollowerProblem}
		P_F(x) = \begin{cases}
			\displaystyle\min_{y}\,\,& \langle c,y\rangle\\
			\mbox{s.t.}\quad & Ax + By \leq b,
		\end{cases}
	\end{equation}
	for matrices $A$ and $B$, and vectors $b$ and $c$ of appropriate dimensions. The main difficulty that arises from the reaction map in this setting is that the affine dimension of the solution of a parametric linear programming problem is not necessarily constant.

	Here, our aim is to show that the reaction map $S:X\tto Y$ of problem \eqref{eq:LinearFollowerProblem} verifies rectangular continuity, and therefore the Bayesian approach is well-posed for beliefs admitting a continuous density function, by invoking Theorems \ref{thm:ExistenceSingleLeader} and \ref{thm:WeakContinuityUnderRecContinuity}. In fact, we will derive our result for a more general class of problems, defined by convex weakly analytic functions (see, e.g., \cite{bank1982non}). Let us recall this definition.
	
	\begin{definition}[Weakly analytic function]\label{def:WeaklyAnalytic}
		A function $g:\R^p\to \R$ is said to be \emph{weakly analytic} if the following holds for any two vectors $y,u\in \R^p$: If the function $g_{y,u}(\alpha)=g(y+\alpha u)$ is constant on an open interval, then $g(y+\alpha u) =g(y)$ for all $\alpha\in\R$.
	\end{definition}
	
	Important examples of convex weakly analytic functions are linear functions and more generally all positive semi-definite  quadratic forms. A useful property of this class is given by the following lemma.
	
	\begin{lemma}\label{lemma:wAnalytic-ConstantConvex}
		Let $g:\R^p\tto\R$ be a convex and weakly analytic function. If $g$ is constant over a convex set $C$, then
		\begin{equation}\label{eq:wAnalytic-ConstantConvex}
			g(y+v)=g(y),\quad \mbox{ for all } y\in \R^p \text{ and } v\in L=\spn(C-C).
		\end{equation}
	\end{lemma}
	\begin{proof}
		Fix $y\in \R^p$ and $v\in L$. We note that from the linear structure of the spaces it is enough to prove that $g(y+v)\leq g(y)$. Choose $\bar{c}\in \ri(C)$. Clearly, $L = \spn(C-\bar{c})$. Since $0\in \ri(C-\bar{c})$, there exists $\bar{\alpha}>0$ such that the open interval $(-\bar{\alpha} v,\bar{\alpha} v)$ is contained in $C-\bar{c}$. Then, the mapping
		$\alpha\mapsto g(\bar{c} + \alpha v)$
		is constant in the interval $(0,\bar{\alpha})$, which yields that
		\[
		g(\bar{c} + \alpha v) = g(\bar{c}),  \quad \forall \alpha\in\R.
		\]
		Now, consider the sequence of points $z_n=\frac{1}{n}(\bar{c}+nv)+ \frac{n-1}{n}y$ which clearly converges to $y+v$. From the convexity of $g$ we deduce that
		\[
		g(z_n)\leq \frac{1}{n}g(\bar{c}+nv)+ \frac{n-1}{n}g(y)= \frac{1}{n}g(\bar{c})+\frac{n-1}{n}g(y).
		\]
		So, passing to the limit, we obtain the desired inequality 
		$ g(y+v)\leq g(y)$. 
	\end{proof}
	
	
	Let $S:X\tto Y$ be the reaction map of the following convex parametric optimization problem
	\begin{equation}
		\min_{y} \{g_0(y)\, :\, g_i(y)\leq {\varphi_i(x)}, i=1,...,k\}
		\label{eqconvexproblem}
	\end{equation}
	where, for each $i=0,...,k$, $g_i:\R^p\to \R$ is convex and weakly analytic  and $\varphi_i$ is continuous, and the feasible set $K(x)=\{y\in \R^p: g_i(y)\leq \varphi_i(x),i=1,...,k\}$ is assumed nonempty for each $x\in X$. Here, the ambient space is $Y = \cco(K(X))$.
	
	Let $\varphi_0(x)=\inf\{g_0(y):y\in K(x)\}$. In what follows, it will be useful to describe $S$ as
	\[
	S(x)=\{y\in \R^p\, :\, g_i(y)\leq \varphi_i(x),\forall i=0,...,k\}.
	\]
	It is known that the weakly analytic property of $\{g_i\ :\ i=0,\ldots,k\}$ entails that the value function $\varphi_0$ and 
	the set-valued map $S$ are both continuous (see e.g. \cite[Theorem 4.3.5]{bank1982non}). The following theorem shows that, assuming that the ambient space $Y$ is compact,  $S$ is in fact rectangularly continuous.
	
	
	\begin{theorem}\label{thm:RecContinuityWeakAnalytic}
		Suppose that the ambient space $Y = \cco(K(X))$ is nonempty and compact. Then the set-valued map $S$ described as the solution of a the parametric problem \eqref{eqconvexproblem}, where for each $i=0,...,k$, $g_i:Y\to \R$ is convex and weakly analytic  and $\varphi_i:X\to\R$ is continuous, is rectangularly continuous.
	\end{theorem}
	\begin{proof}
		Since $S$ is lower semicontinuous with compact values, Michael Selection Theorem \cite{michael1956continuous} ensures that there exists a continuous selection $s:X\to Y$ of $S$. It is direct that rectangular continuity of $S$ is equivalent to rectangular continuity of $x\tto S(x) - s(x)$. Then, 
		by replacing $S(x)$ with $S(x) - s(x)$, we may and do assume that $0\in S(x)$ for all $x\in X$.
		
		Let $\bar{x}\in X$ and $L = \spn(S(\bar{x}))$. We define $R_0$ and $R_1$ as the same set-valued map $R$ given by 
		\[ 
		R(x)=\{y_r\in L^{\perp}: (L+y_r)\cap S(x)\neq\emptyset \}.
		\]
		Let us first verify that $\Lim_{x\to\bar{x}}R(x)=\{0\}$. Indeed, let $(x_n)\subset X$ be any sequence converging to $\bar{x}$, and let $y=\lim y_n$ with $y_n\in R(x_n)$. Since $y_n\in R(x_n)$, then there exists $z_n\in L$ such that $z_n+y_n\in S(x_n)$. Since $S$ is upper semicontinuous and the open set $U=L+B(\varepsilon)$ (for $\varepsilon>0$) contains $S(\bar{x})$, then there exists $n_0$ such that $n\geq n_0$ implies that $z_n+y_n\in L+B(\varepsilon)$.
		Projecting onto $L^{\perp}$, we deduce that $y_n\in B(\varepsilon)\cap L^{\perp}$. Now, by the continuity of the projection onto $L^{\perp}$ and the arbitrariness of $\varepsilon$, we conclude that $y_n\to 0$. This proves that $\Limsup_{x\to\bar{x}}R(x)=\{0\}$, which entails, by noting that $0\in R(x)$ for all $x\in X$, that $\Lim_{x\to\bar{x}}R(x)=\{0\}$.
		
		Let us define the maps $T_0$ and $T_1$ as
		\begin{equation*}
			T_0(x)=\bigcap_{y_r\in R(x)}S(x)-y_r,\qquad T_1(x)=\left(\bigcup_{y_r\in R(x)}S(x)-y_r\right)\cap L.
		\end{equation*}
		On the one hand, we always have that $T_0(x)\subset L$ and this set corresponds to the portion of $L$ that is contained in all the fibers $(L+y_r)\cap S(x)$, with $y_r\in R(x)$. On the other hand, $T_1(x)$ coincides with the projection of $S(x)$ onto $L$. Figure \ref{fig:T0_and_T1} illustrates this construction.
		
		\begin{figure}[h]
			\centering
			\scalebox{1.2}{
				\begin{tikzpicture}[scale=0.85]
					\draw[thick] (-2.5,0) -- (-0.5,2) -- (0.5,2) -- (2.5,0) -- cycle;
					\draw (-3.5,1)--(3.5,1);
					\node[above] at (0,2) {\small$S(x)$};
					\node[right] at (3.5,1) {\small $L$};
					\draw[dashed] (-2.5,0)--(-2.5,2)--(-0.5,2);  
					\draw[dashed] (2.5,0)--(2.5,2)--(0.5,2);
					\draw[very thick](-2.5,1)--(2.5,1);
					\draw[dashed] (-0.5,0)--(-0.5,2);
					\draw[dashed] (0.5,0)--(0.5,2);
					\draw [decorate, decoration={brace,amplitude=2pt,aspect=0.1}, xshift=0.4pt, yshift=0.4pt](-2.5,1.1)--(2.5,1.1) node[black,near start,xshift=-0.7cm,yshift=0.3cm] {\small $T_1(x)$};
					\draw [decorate, decoration={brace,amplitude=2pt, mirror}, xshift=0.4pt, yshift=-0.4pt](-0.5,0.9)--(0.5,0.9) node[black,midway,yshift=-0.3cm] {\small $T_0(x)$};
				\end{tikzpicture}
			}
			\caption{Illustration of $T_0(x)$ and $T_1(x)$}
			\label{fig:T0_and_T1}
		\end{figure}
		
		Since $R(\bar{x})= \{0\}$ and $S(\bar{x})\subset L$, it is not hard to realize that $T_0(\bar{x}) =S(\bar{x}) = T_1(\bar{x})$. Furthermore, by construction we have
		\begin{equation}
			T_0(x)+R(x) \subset S(x)\subset T_1(x)+R(x),\quad \forall x\in X.
			\label{eqinclusions}
		\end{equation}
		
		We will prove next that $\Lim_{x\to\bar{x}} T_j(x)=T_j(\bar{x})=S(\bar{x})$ for both $j=0$ and $j=1$.
		Since $R(x)$ obviously contains 0, then it is quite simple to deduce from \eqref{eqinclusions} that $\Liminf_{x\to\bar{x}} T_1(x)\supset S(\bar{x})$ and $\Limsup_{x\to\bar{x}}T_0(x)\subset S(\bar{x})$. So, it is enough to prove that
		\begin{equation}\label{eq:InclusionsToProve}
			\Limsup_{x\to\bar{x}}T_1(x)\subset	S(\bar{x})\subset\Liminf_{x\to\bar{x}} T_0(x). 
		\end{equation}
		
		The first inclusion is quite simple. Choose any sequence $(x_n)$ converging to $\bar{x}$ and any convergent sequence $(y_n)$ with $y_n\in T_1(x_n)$ for each $n\in \N$. Let $y = \lim y_n$. By construction, there exists a sequence $(z_n)$ with $z_n\in R(x_n)$ such that $y_n+z_n\in S(x_n)$ for each $n\in \N$.
		Since $\Lim R(x_n) = \{ 0\}$, then $z_n\to 0$ and so, $y_n+z_n\to y$. Moreover, since $S$ is upper semicontinuous, we conclude that $y\in S(\bar{x})$, which verifies the desired inclusion.
		
		For the second inclusion of \eqref{eq:InclusionsToProve}, we only need to show that $T_0$ is lower semicontinuous at $\bar{x}$, since $T_0(\bar{x}) = S(\bar{x})$. Take first $\tilde{y}\in \ri\ (S(\bar{x}))$. We will prove that $\tilde{y}\in T_0(x)$ for $x$ near enough to $\bar{x}$. 
		
		Let us define $I=\{i=0,...,k: g_i(y)={\varphi_i(\bar{x})},\ \forall y\in S(\bar{x})\}$ and $J = \{1,\ldots,k\}\setminus I$. 
		By \cite[Lemma 3.2.1]{bank1982non}, the inclusion $\tilde{y}\in \ri(S(\bar{x}))$ ensures that 
		\[
		g_i(\tilde{y})<{\varphi_i(\bar{x})},\;  i\in J.
		\]
		Fix $i\in J$. By continuity of $g_i$ and $\varphi_i$, there exist $\varepsilon_i,\delta_i>0$ such that for each $x\in B(\bar{x},\delta_i)\cap X$ and each $y\in B(\tilde{y},\epsilon_i)$, one has that $g_i(y)<\varphi_i(x)$.
		
		Take $\varepsilon = \min\{\varepsilon_i\ :\ i\in J \}$ and choose $\delta \leq \min\{\delta_i\ :\ i\in J \}$ small enough, such that the inclusion $R(x)\subset B(0,\varepsilon)$ is verified for each $x\in B(\bar{x},\delta)\cap X$. Such $\delta>0$ exists thanks to  the continuity of $R$ at $\bar{x}$.
		
		Now, let $x\in B(\bar{x},\delta)\cap X$ and let $y_r\in R(x)$. On the one hand, the above development ensures that
		\[
		g_i(\tilde{y} + y_r) < \varphi_i(x),\qquad \forall i\in J.
		\]
		
		On the other hand, since $y_r\in R(x)$, there exists $y_1\in L$ such that $y_1+y_r\in S(x)$. Noting that for any $i\in I$, $g_i$ is constant on $S(\bar{x})$ and recalling that $L = \spn(S(\bar{x}))$ and that $g_i$ is convex and weakly analytic, Lemma \ref{lemma:wAnalytic-ConstantConvex} yields that  
		\begin{align*}
			g_i(\tilde{y}+y_r)&=g_i(y_1+y_r+\overbrace{\tilde{y}-y_1}^{\in L})= g_i(y_1+y_r)\leq\varphi_i(x),\quad\forall i\in I.
		\end{align*}
		
		Mixing both inequalities, we deduce that  $\bar{y}+y_r\in S(x)$. Since $y_r\in R(x)$ is arbitrary, we deduce that
		\[ 
		\tilde{y}\in\bigcap_{y_r\in R(x)}S(x)-y_r=T_0(x), \quad \forall x\in B(\bar{x},\delta)\cap X.
		\]
		We conclude directly that $\tilde{y}\in \Liminf_{x\to\bar{x}}T_0(x)$, and so, $\ri (S(\bar{x}))\subset \Liminf_{x\to\bar{x}}T_0(x)$. Noting that $\Liminf_{x\to\bar{x}}T_0(x)$ is closed and recalling that $S(\bar{x})$ is convex, we can write
		\[ 
		S(\bar{x})=\cl(\ri\ (S(\bar{x}))) \subset \Liminf_{x\to\bar{x}}T_0(x),
		\]
		which finishes the proof.
	\end{proof}
	

	
	
	Finally, we can state a direct corollary, which provides a positive answer to the main motivation of this work. Its extension to the case of multiple leaders can be easily done to ensure existence of a mixed equilibrium, by invoking Theorem \ref{thm:ExistenceSingleLeader}.

	\begin{corollary}\label{cor:ExistenceLinear} Any bilevel programming problem of the form \eqref{eq:BilevelProgrammingProblem}, whose lower-level problem is of the form \eqref{eq:LinearFollowerProblem}, admits a solution under the Neutral approach (or under the Bayesian approach with any belief admitting a continuous strictly positive density function), provided that
		\begin{enumerate}
			\item[(i)] The decision set $X$ of the leader is compact and the cost function $\theta$ is lower semicontinuous.
			\item[(ii)] The constraint set $K(x):=\{y\in \R^p: Ax+By\leq b\}$ is bounded for each $x\in X$.
			\item[(iii)] $X\cap\dom K\neq\emptyset$.
		\end{enumerate}
	\end{corollary}

	
	\section{Numerical experiments for the Neutral approach}\label{sec:Computability}
	
	This section is devoted to provide a numerical method to estimate solutions of bilevel programming problems under the Bayesian approach. Since this is the first work (to the best of our knowledge) dealing with numerical methods in a general form, we restrict ourselves to the study of bilevel programming with linear lower-level in the Neutral approach. So let us consider the reaction map associated to a linear parametric optimization problem of the form
	\begin{equation}
		S(x):=\argmin_y \{\langle c,y\rangle: Ax+By\leq b\},
		\label{eq:sofx}
	\end{equation}
	where $c\in \R^p$, $b\in\R^k$ and $A$, $B$ are matrices of appropriate dimensions. We assume that the feasible region $\mathscr{F}=\{(x,y)\in \R^d\times \R^p: Ax+By\leq b \}$ is nonempty and bounded, that is, a polytope. 
	We consider a leader whose problem is defined by $X=\dom(S) = \{ x\in\R^d\, :\ \exists y\in\R^p,\, (x,y)\in \mathscr{F}  \}$ and a continuous objective function $\theta:\R^d\times\R^p\to \R$. The ambient space is put as $Y = \{ y\in\R^p\, :\ \exists x\in\R^d,\, (x,y)\in \mathscr{F}  \}$. The linear bilevel programming problem in the Neutral approach is then 
	\begin{equation}\label{eq:ProblemLeader-Neutral}
		\min_{x\in X} \varphi
		(x)= \E_{\iota_x}(\theta(x,\cdot)),
	\end{equation}
	where $\iota:X\tto Y$ is the Neutral belief over $S$. 
	Note that if the leader has supplementary constraints of the form $Cx\leq e$, they can be incorporated into $S$ in the system $(A,b)$, and so there is no loss of generality in considering $X$ as the domain of $S$.
	
	One might guess that if $\theta(x,y)$ is also linear, then a solution of \eqref{eq:ProblemLeader-Neutral} can be found as the projection of a vertex of the feasible region $\mathscr{F}$, as it is the case for the Optimistic and the Pessimistic approaches (see, e.g., \cite{bialas1984two,zheng2016solution}). The following example shows, however, that the Neutral approach does not necessarily verify this property.
	
	\begin{example}
		\label{ex:neutralsolnotvertex}
		Consider the linear bilevel programming problem where the objective of the leader is defined as $\theta(x,y) = \langle d_1,x\rangle + \langle d_2,y\rangle$ with  $d_1=(0,0)$ and $d_2=(1,-7)$, and the reaction map $S$ is defined as the minimal convex-graph set-valued map such that 
		$S(v_1)=S(v_3)=\co\{(0,0),(1,0),(0,1)\}$ and $
		S(v_2)=S(v_4)=\co\{(1,0),(2,0),(2,1),(1, 1)\}$,
		where $v_1=(1,0)$, $v_2=(0,1)$, $v_3=(-1,0)$ and $v_4=(0,-1)$ (see Figure \ref{fig:examplehyperbanana}).  
		Equivalently, $S(x)$ can we written in the form \eqref{eq:sofx} by considering $c=(0,0)$ and
		$$
		\scriptsize
		A=
		\begin{pmatrix}
			-1 & -1 \\
			-1 &  1 \\
			-2 &  0 \\
			-1 &  0 \\
			0 & -1 \\
			0 &  0 \\
			0 &  0 \\
			0 &  1 \\
			1 &  0 \\
			2 &  0 \\
			1 & -1 \\
			1 &  1 \\
		\end{pmatrix},\, 
		B=\begin{pmatrix}
			0 &  0 \\
			0 &  0 \\
			1 &  1 \\
			1 &  0 \\
			-1 &  0 \\
			0 & -1 \\
			0 &  1 \\
			-1 &  0 \\
			1 &  0 \\
			1 &  1 \\
			0 &  0 \\
			0 &  0 \\
		\end{pmatrix},\, 
		b=\begin{pmatrix}
			1 \\
			1 \\
			3 \\
			2 \\
			0 \\
			0 \\
			1 \\
			0 \\
			2 \\
			3 \\
			1 \\
			1 \\
		\end{pmatrix}.
		$$
		It is easy to show that $\varphi(v_1) = \varphi(v_3) = -2$ and that $\varphi(v_2) = \varphi(v_4) = -2.25$. However, by taking $v_0 = (0,0)$, which is not the projection of any vertex of the feasible region, we get a better value, namely $\varphi(v_0) = -2.5$.
		
		\begin{figure}[h]
			\begin{center}
				\begin{tikzpicture}[scale=0.65]
					
					\draw[->] (-5,0) -- (5,0) node[anchor=south] {$x_1$};
					\draw[->] (-2,-2) -- (2,2) node[anchor=south] {$x_2$};
					
					\coordinate (v1) at (-3,0);
					\coordinate (v2) at (3,0);
					\coordinate (v3) at (-1.5,-1.5);
					\coordinate (v4) at (1.5,1.5);
					\foreach \x in {(v1), (v2), (v3), (v4)}{
						\fill \x circle[radius=2pt];
					}
					
					\node[below of = v1, node distance = .3cm] {$v_3$};
					\node[below of = v2, node distance = .3cm] {$v_1$};
					\node[below right of = v3, node distance = .3cm] {$v_4$};
					\node[below right of = v4, node distance = .3cm] {$v_2$};

					\coordinate (O1) at (-2,3);
					\draw[dashed] (v1) to[out=40,in=-70] (O1);
					\draw[dashed] (v2) to[out=155,in=-70] (O1);
					
					\coordinate (X1) at (.5,3);
					\coordinate (Y1) at (-2,4.5);
					
					\draw[->] (O1) -- (X1) node[anchor=south] {$y_1$};
					\draw[->] (O1) -- (Y1) node[anchor=south] {$y_2$};
					\filldraw[fill=gray!30!white] (O1) -- (-1,3) -- (-2,4) -- cycle;
					
					\coordinate (O2) at (3,3);
					\draw[dashed] (v3) to[out=30,in=-70] (O2);
					\draw[dashed] (v4) to[out=20,in=-70] (O2);
					
					\coordinate (X2) at (5.5,3);
					\coordinate (Y2) at (3,4.5);
					
					\draw[->] (O2) -- (X2) node[anchor=south] {$y_1$};
					\draw[->] (O2) -- (Y2) node[anchor=south] {$y_2$};
					\filldraw[fill=gray!30!white] (4,3) -- (5,3) -- (5,4) -- (4,4) -- cycle;
				\end{tikzpicture}
			\end{center}
			\caption{Values of $S$ of Example \ref{ex:neutralsolnotvertex} at the projections of the vertices of its graph}
			\label{fig:examplehyperbanana}
		\end{figure}
	\end{example}
	
	In order to solve \eqref{eq:ProblemLeader-Neutral}, it is necessary to compute the neutral value function $\varphi(x) = \E_{\iota_x}(\theta(x,y))$. This computation can be difficult since $\varphi$ is given as an integral over a parametric set. 
	However, assuming that $\theta$ is continuous, $\varphi(x)$ can be easily approximated through Monte-Carlo estimations (see, e.g., \cite{LeobacherPillichshammer2014}), by embedding $S(x)$ into a space of dimension equal to $\adim(S(x))$.

	Example \ref{ex:neutralsolnotvertex} is very illustrative in the above discussion, and thus, we conducted a detailed study of it.
	The first step is to observe that, since the leader's objective is of the form $\theta(x,y) = \langle d_1,x\rangle + \langle d_2,y\rangle$, then we have that
	$\varphi(x) = \left\langle d_1,x\right\rangle+ \left\langle d_2,\centroid(S(x))\right\rangle$.
	Next, since the reaction map $S$ satisfies that
	$S(x_1,x_2) = S(|x_1|,|x_2|)$,
	it is enough to compute $\varphi(x)$ in the positive quadrant part of its domain, which is given by the set $\Delta = \co\{ (0,0),(1,0),(0,1) \} = \co\{ v_0,v_1,v_2 \}$. In the case of Example \ref{ex:neutralsolnotvertex} we can give an explicit formula for $\varphi$ in $\Delta$. Indeed, for every $x=(x_1,x_2)\in \Delta$, the images of $S$ have the following form
	\begin{equation*}
		S(x)=\co\{(x_2,0),(2-x_1,0),(2-x_1,1-x_1),(2-2x_1,1),(x_2,1)\}.
	\end{equation*}
	So, by computing the centroid, we have the following Explicit formula in $\Delta$ for the neutral value function 
	\begin{equation}\label{eq:explicitformula}
		\varphi(x)=\frac{-30+9x_1+18x_1^2-3x_1^3+21x_2-3x_2^2}{12-6x_1-6x_2-3x_1^2}.
	\end{equation}
	
	In Figure \ref{fig:Explicit}, the graph of $\varphi$ of Example \ref{ex:neutralsolnotvertex} is computed using the above explicit formula. In parallel, Figure \ref{fig:MonteCarlo} shows the graph of $\varphi$ obtained by plain Monte-Carlo simulations. While the Monte-Carlo method is noisy, it clearly preserves the general shape of the function $\varphi$, and so, it seems to be a promising method to work with. 
	
	\begin{minipage}{0.5\linewidth}
		\begin{figure}[H]
			\centering
			\includegraphics[scale=0.135]{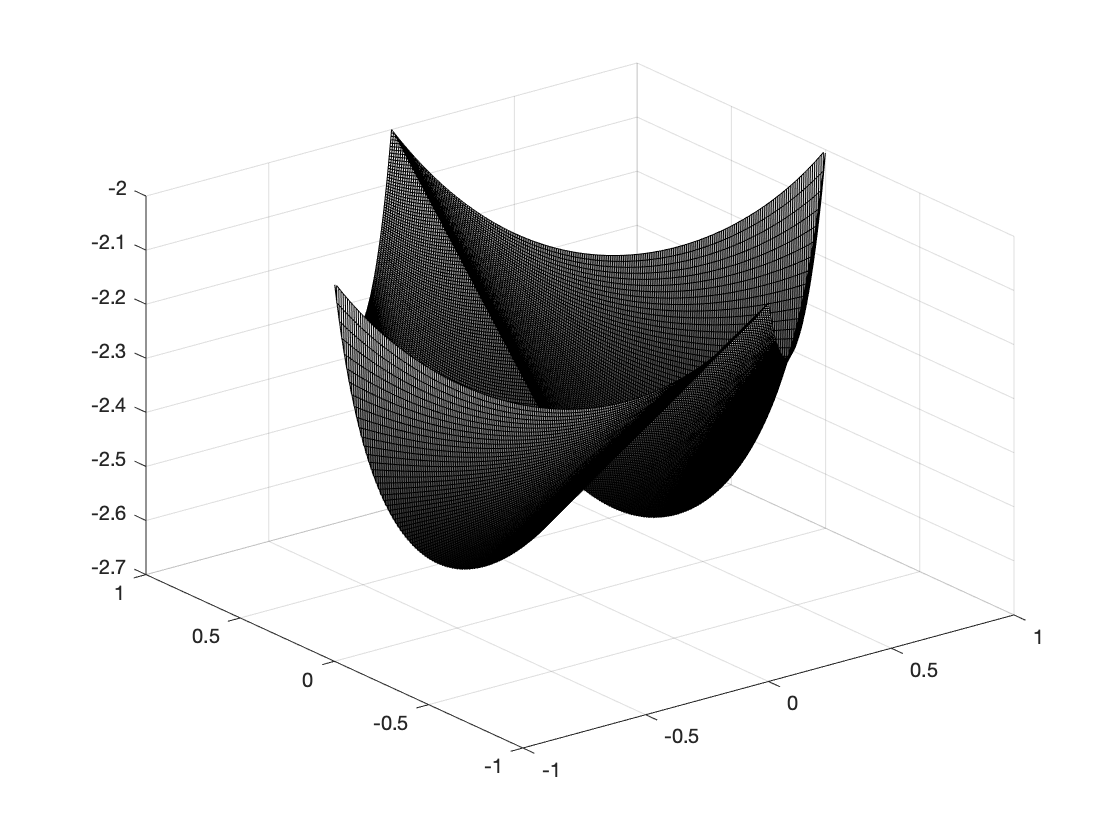}
			\caption{Explicit formula}
			\label{fig:Explicit}
		\end{figure}
	\end{minipage}
	\begin{minipage}{0.5\linewidth}
		\begin{figure}[H]
			\centering
			\includegraphics[scale=0.135]{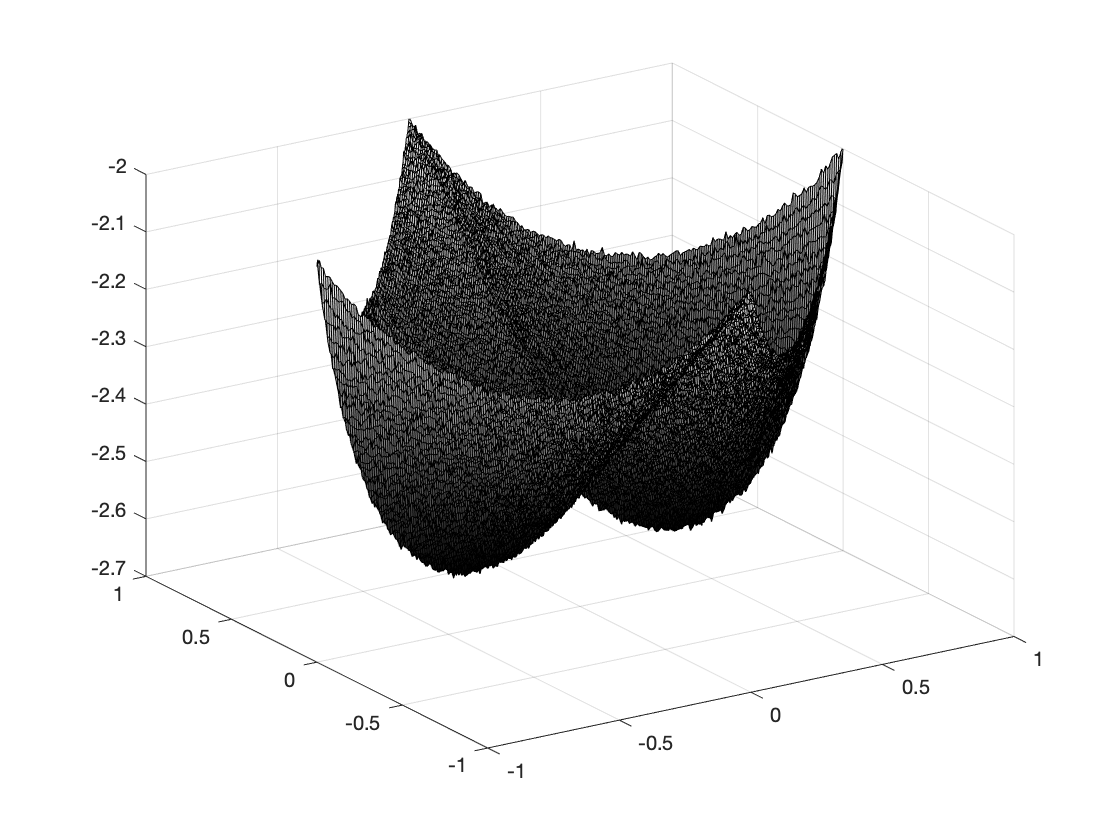}
			\caption{Monte-Carlo method}
			\label{fig:MonteCarlo}
		\end{figure}
	\end{minipage}
	
	
	To find the solution of the neutral approach we used the Differential Evolution metaheuristic (see, e.g., \cite{Yang2014}). The use of evolutionary metaheuristics in bilevel optimization is not new (see, e.g., \cite{Dempe2020Bibliography} and the references therein). However, to the best of our knowledge, this is the first time that it is 
	used to estimate the solutions of the Bayesian approach. Our method is described in Algorithm \ref{alg:DiffEvol}.
	
	\begin{algorithm}
		\caption{Differential Evolution for bilevel programming with linear lower-level under neutral approach}
		\label{alg:DiffEvol}
		\begin{algorithmic}[1]
			\REQUIRE \small{Linear problem of the follower $(A,B,b,c)$ with nonempty bounded feasible region $\mathscr{F} = \{(x,y)\in\R^d\times\R^p:Ax+By\leq b\}$; Objective function of the leader $\theta$ continuous; Sampling size $N$.}
			\STATE{\small Compute $X = \proj_{\R^d}(\mathscr{F})$.}
			\STATE{\small To evaluate $\E_{\iota_x}(\theta(x,\cdot))$ we do two steps:}
			\bindent
			\STATE{\small Embed $S(x)$ in a space of dimension $\adim(S(x))$ and find a bounded box $B$ containing $S(x)$.}
			\STATE{\small Compute $\E_{\iota_x}(\theta(x,\cdot))$ using Monte-Carlo simulations in the embedded box $B$} with $N$ samplings.
			\eindent
			\STATE{\small Run the Differential Evolution metaheuristic with the objective function $\varphi(x) = \E_{\iota_x}(\theta(x,\cdot))$ and the constraint $x\in X$.}
			\RETURN \small{$x^*$ as the Output of the Differential Evolution metaheuristic.}
		\end{algorithmic}
	\end{algorithm}
	
	The numerical experiments were done in \texttt{Matlab R2019b} \cite{MATLAB:R2019b}  using the open-source implementation of M. Buehren \cite{DiffEvolutionMatlab} for the Differential Evolution metaheuristic. We tested the algorithm with 21 linear bilevel problems of the BOLIB library (see \cite{zhou2020bolib} for a description of the library) and the new academic problem of Example \ref{ex:neutralsolnotvertex}. We also tested nonlinear bilevel problems from BOLIB library with linear lower-level problems (problems 24 to 27 in Table \ref{table:bolibneutral}). We selected only the problems of BOLIB with bounded feasible region. For problems with coupling constraints, we modified them passing the constraints to the follower (which is the case only for Problem 26). The sampling size was fixed at $N=10^6$. Results are summarized in Table \ref{table:bolibneutral}.  
	
	\begin{table}[h!]
		\begin{center}
			\scriptsize
			\begin{tabular}{c|l|c|c|c|r|r}
				\hline
				ID & Reference& $n_x$&$n_y$&$n_g$ & $\varphi^o$ &$\varphi$ (DE)\\
				\hline
				1&AnandalinghamWhite1990  &   1&1& 7 & -49 & -49\\
				2&Bard1984a               &   1&1& 6 & 28/9 & 3.111111\\
				3&Bard1984b               &   1&1& 6 & -37.6 & -37.6\\
				4&Bard1991Ex2             &   1&2 &6 & -1 & -1\\
				5&BardFalk1982Ex2         &   2&2& 7 & -3.25 &-3.25\\
				6&Ben-AyedBlair1990a      &   1&2& 6 & -2.5 & -2.5 \\
				7&Ben-AyedBlair1990b      &   1& 1& 5 & -6 & -6 \\
				8&BialasKarwan1984a       &   1&2& 8 & -2 &-2 \\
				9&BialasKarwan1984b       &   1&1& 7 & -11 &-11 \\
				10&CandlerTownsley1982     &   2&3& 8 &-29.2 & -29.2 \\
				11&ClarkWesterberg1988     &   1&1& 3 & -37 & -37 \\
				12&ClarkWesterberg1990b    &   1&2& 7 & -13  &-9.5062 \\
				13&GlackinEtal2009         &   2&1& 6 & 6 & 6.0171\\
				14&HaurieSavardWhite1990   &   1&1& 4 & 27 & 27\\
				15&HuHuangZhang2009        &   1&2& 6 & -76/9 &-8.7778 \\
				16&LanWenShihLee2007       &   1&1& 8 & -85.09 & -85.0909 \\
				17&LiuHart1994             &   1&1& 5 & -16 &-16\\
				18&TuyEtal1993             &   2&2& 7 & -3.25 & -3.25 \\
				19&TuyEtal1994             &   2&2& 6 & 6 &6\\
				21&VisweswaranEtal1996     &   1&1& 6 & 28/9& 3.111122 \\
				22&Example \ref{ex:neutralsolnotvertex}    & 2&2&12& -7 & -2.61223 \\
				\hline
				\hline
				23&BardBook1998Ex832    & 2&2&11& 0 & 2.13643e-09 \\
				24&LamparielloSagratella2017Ex35 &1&1&5& 0.8 & 0.8   \\
				25&TuyEtal2007*& 1& 1& 5 & 22.5 & 22.5\\
				26&WanWangLv2011**&2&3&8& - & 7.500037\\ 
				\hline
			\end{tabular}
		\end{center}
		\caption{\small Optimistic value vs Neutral value of solutions for the bilevel problems with linear lower-level of BOLIB and Example \ref{ex:neutralsolnotvertex}. Problems 23-26 have nonlinear objective functions for the leader. $n_x, n_y$ and $n_g$ are the x-dimension, y-dimension, and number of constraints, respectively. $\varphi^o$ is the registered optimistic value (see \cite{zhou2020bolib}). $\varphi$ has the value given by Algorithm \ref{alg:DiffEvol}.	}
		\label{table:bolibneutral}
	\end{table}
	An important observation is that for any bilevel programming problem, if there exists an optimistic solution $x$ such that $S(x)$ is a singleton, then any optimistic solution is also a Bayesian solution for any belief of the leader, in particular, for the Neutral belief. This is what occurs in most problems of BOLIB, and so, with the exception of problem 12, the exact solution of the neutral approach coincides with the value of $\varphi^o$. 
	
	In the case of Example \ref{ex:neutralsolnotvertex}, we also run a classic interior-point algorithm (implemented in the method \texttt{fmincon} of \texttt{Matlab R2019b}) in the positive quadrant, using the explicit formula of \eqref{eq:explicitformula}. We obtained the point $\bar{x} = (0.391,0)$ as a neutral solution, with neutral value $\varphi(\bar{x}) = -2.6001$. The result of the DE heuristic for the same problem was $\bar{x} = (-0.3954,0.0003)$ with value $\varphi(\bar{x}) = -2.61223$. Recalling the symmetry of the problem, the results are very similar, with an error of $0.5\%$. The errors come from the heuristic itself, but also from the noise of Monte-Carlo estimation. The overestimated value $-2.61223$ is referable to this noise and the fact that the DE metaheuristic preserves the best value, including overestimation. 
	
	

	\section{Final comments and perspectives} \label{sec:Final}
	
	In this work we extended previous existence results for bilevel problems in the Bayesian approach admitting change of dimension beyond the  
	case of full dimension (nonempty interior of values) and null dimension (single-valued).
	Our existence results are based on the weak continuity of the beliefs (as decision-dependent distributions) and the rectangular continuity property, which allows to control the possible changes of affine dimension of the reaction map of the followers. Our results cover regularized problems (Corollary \ref{cor:ExistenceRegularized}) but also (non-regularized) problems with linear lower-level (Corollary \ref{cor:ExistenceLinear}) or, more generally, with separable 
	and convex weakly-analytic data in the sense of formulation \eqref{eqconvexproblem}. The bilevel game model we consider does not include coupling constraints, which we think would be an interesting future development of the subject.
	
	The condition of rectangular continuity introduced in this work might seem a bit technical but really served as a tool to prove Corollary \ref{cor:ExistenceLinear}, which was the main motivation of the work.
	
	An open question left by this work is whether rectangular continuity over the reaction map $S$ is in some sense minimal in order to ensure that beliefs with continuous strictly positive densities over $S$ are weak continuous. Otherwise, it would be interesting to show a counterexample or more importantly to find a more general (or an independent) condition than rectangular continuity enabling us to prove weak continuity of beliefs over $S$.
	
	We tested the Neutral approach (Bayesian approach with neutral belief) using a Differential Evolution metaheuristic for a family of test problems whose lower-level is linear, since we know that in this case the neutral value function is continuous. The evaluation of the value function of the neutral approach was estimated using plain Monte-Carlo integration. The obtained results are promising, since in most cases, we could compare the results with the theoretical exact solution, obtaining quite similar values. In a future work, we will further develop the algorithmic aspect of the Bayesian approach. As hinted in the Example \ref{ex:neutralsolnotvertex}, it seems that there is a stronger regularity of the value function for linear problems that needs to be explored. 
	
	The development of algorithms or heuristics to tackle the Bayesian approach remains as a perspective for this work. In particular, we would like to develop algorithms to compute solutions for both the case with one leader and the case with multiple leaders.
	
	\section*{Acknowledgements}
	We thank L. Mallozzi for letting us know about her work with J. Morgan \cite{Mallozzi1996}, which we were not aware of when the first preprint version of this work was communicated. We are also grateful to S. Dempe for his comments on the first version.

	\bibliographystyle{plain}
	\bibliography{Biblio2.bib}
\end{document}